\documentclass[twoside,leqno]{amsart}

\usepackage{amsmath,color,amsthm}

\makeatletter
\@namedef{subjclassname@2010}{%
  \textup{2010} Mathematics Subject Classification}
\makeatother

\numberwithin{equation}{section}

\newtheorem{theorem}{Theorem}[section]
\newtheorem{corollary}[theorem]{Corollary}
\newtheorem{lemma}[theorem]{Lemma}
\newtheorem{proposition}[theorem]{Proposition}

\newtheorem{example}[theorem]{\sl Example}
\newtheorem{definition}[theorem]{\sl Definition}

\theoremstyle{definition}
\newtheorem{Remark}[theorem]{Remark}

\newcommand{\beqn}{\begin{eqnarray}}
\newcommand{\eeqn}{\end{eqnarray}}
\newcommand{\beqnn}{\begin{eqnarray*}}
\newcommand{\eeqnn}{\end{eqnarray*}}

\newcommand{\SV}{S_n^{\mbox{\scriptsize V}}}
\newcommand{\SQ}{S_n^{\mbox{\scriptsize Q}}}
\newcommand{\SQnk}{S_{n,k}^{\mbox{\scriptsize Q}}}
\newcommand{\LQk}{L_k(n)}
\newcommand{\LQkm}{L_{k-1}(n)}
\newcommand{\pivrank}{\mbox{pivrank}_k(n)}

\newcommand{\RQk}{R_k(n)}
\newcommand{\RQkm}{R_{k-1}(n)}

\newcommand{\rank}{m_n}
\newcommand{\pSeedk}{U_{\tau_k(n)}}

\newcommand{\nukm}{\nu_{k-1}(n)}

\def\cal{\mathcal}

\newcommand{\lf}{\left\lfloor}

\newcommand{\rf}{\right\rfloor}

\newcommand{\EE}{{\bf  E}}

\newcommand{\PP}{{\bf  P}}

\newcommand{\Dt}{{\tilde{D}}}
\newcommand{\Ht}{{\widetilde{H}}}

\newcommand{\St}{{\widetilde{S}}}

\newcommand{\Lc}{{\mathcal L}}

\newcommand{\Leq}{{\,\stackrel{\Lc}{=}\,}}

\newcommand{\asto}{{\stackrel{\mbox{\rm \scriptsize a.s.}}{\longrightarrow}}}
\newcommand{\Lpto}{{\stackrel{L^p}{\longrightarrow}}}

\newcommand{\Hh}{\widehat{H}}

\newcommand{\hti}{\tilde{h}}
\newcommand{\begp}{\begin{proposition}}
\newcommand{\enp}{\end{proposition}}
\newcommand{\begt}{\begin{theorem}}
\newcommand{\ent}{\end{theorem}}
\newcommand{\begl}{\begin{lemma}}
\newcommand{\enl}{\end{lemma}}
\newcommand{\begc}{\begin{corollary}}
\newcommand{\enc}{\end{corollary}}
\newcommand{\begcl}{\begin{claim}}
\newcommand{\encl}{\end{claim}}
\newcommand{\begr}{\begin{Remark}\rm}
\newcommand{\enr}{\end{Remark}}
\newcommand{\begal}{\begin{algorithm}}
\newcommand{\enal}{\end{algorithm}}
\newcommand{\begd}{\begin{definition}}
\newcommand{\enf}{\end{definition}}
\newcommand{\begx}{\begin{example}}
\newcommand{\enx}{\end{example}}
\newcommand{\bega}{\begin{array}}
\newcommand{\ena}{\end{array}}

\newcommand{\sfrac}[2]{{\textstyle\frac{#1}{#2}}}

\def\rompar(#1){\textup(#1\textup)}    

\newcommand{\refS}[1]{Section~\ref{#1}}
\newcommand{\refT}[1]{Theorem~\ref{#1}}

\newcommand{\refL}[1]{Lemma~\ref{#1}}
\newcommand{\refP}[1]{Proposition~\ref{#1}}
\newcommand{\refR}[1]{Remark~\ref{#1}}

\newcommand\eg{e.g.\spacefactor=1000}

\newcommand\cf{{cf.}\spacefactor=1000}

\newcommand\noqed{\renewcommand{\qed}{}} 

\newcommand\QuickSort{\texttt{QuickSort}}

\newcommand\urladdrx[1]{{\urladdr{\def~{{\tiny$\sim$}}#1}}}

\begin{document}

\newcommand{\tab}[0]{\hspace{.1in}}

\title[Convergence in Distribution for QuickSelect]
{Distributional Convergence for the Number of Symbol Comparisons Used by QuickSelect}

\author{James Allen Fill}
\address{Department of Applied Mathematics and Statistics,
The Johns Hopkins University,
3400 N.~Charles Street,
Baltimore, MD 21218-2682 USA}
\email{jimfill@jhu.edu}
\urladdrx{http://www.ams.jhu.edu/~fill/}
\thanks{Research for both authors supported by the Acheson~J.~Duncan Fund for the Advancement of Research in Statistics.}
\author{Takehiko Nakama}
\address{European Centre for Soft Computing, 
Edificio de Investigaci\'{o}n,
Calle Gonzalo Guti\'{e}rrez Quir\'{o}s S/N,
33600 Mieres, Asturias, Spain}
\email{takehiko.nakama@softcomputing.es}
\urladdrx{http://www.softcomputing.es/}
\subjclass[2010]{Primary 60F25; Secondary 68W40}
\keywords{{\tt QuickSelect}, {\tt QuickQuant}, {\tt QuickVal}, limit distribution, $L^p$-convergence, almost sure convergence, symbol comparisons, probabilistic source}

\date{February~10, 2012.  Revised September~20, 2012.}

\maketitle

\begin{center}
{\sc Abstract}
\vspace{.3cm}
\end{center}
When the search algorithm {\tt QuickSelect} compares keys during its execution in order to find a key of target rank, it must operate on the keys' representations or internal structures, which were ignored by the previous studies that quantified the execution cost for the algorithm in terms of the number of required key comparisons.  In this paper, we analyze running costs for the algorithm that take into account not only the number of key comparisons but also the cost of each key comparison.  We suppose that keys are represented as sequences of symbols generated by various probabilistic sources and that {\tt QuickSelect} operates on individual symbols in order to find the target key. We identify limiting distributions for the costs and derive integral and series expressions for the expectations of the limiting distributions.  These expressions are used to recapture previously obtained results on the number of key comparisons required by the algorithm.

\section{Introduction and Summary}\label{S:introduction}
{\tt QuickSelect}, introduced by Hoare~\cite{h1961} in~1961 and also known as {\tt Find} or ``Hoare's selection algorithm'', is a simple search algorithm widely used for finding a key (an object drawn from a linearly ordered set) of target rank in a file of keys.
We briefly review the operation of the algorithm.  Suppose that there are~$n$ keys (we will suppose that these are all distinct) and that the target rank is~$m$, where $1 \leq m \leq n$.
{\tt QuickSelect}\ $\equiv$\  {\tt QuickSelect}$(n, m)$ chooses a uniformly random key, called the pivot, and compares each other key to it.  This determines the rank~$j$ (say) of the pivot.  If $j = m$, then the algorithm returns the pivot key and terminates.  If $j > m$, then {\tt QuickSelect} is applied recursively to find the key of rank~$m$ in the set of $j - 1$ keys found to be smaller than the pivot.  If $j < m$, then {\tt QuickSelect} is applied recursively to find the key of rank $m - j$ in the set of $n - j$ keys larger than the pivot.

Many studies have examined this algorithm to quantify its execution costs (a non-exhaustive list of references is 
Knuth~\cite{k1972};
Mahmoud, Modarres, and Smythe~\cite{mms1995};
Prodinger~\cite{p1995}; Gr\"{u}bel and R\"{o}sler~\cite{gr1996}; Lent and Mahmoud~\cite{lm1996}; 
Gr\"{u}bel~\cite{g1998}; Mahmoud and Smythe~\cite{ms1998}; Devroye~\cite{d2001}; Hwang and Tsai~\cite{ht2002}; 
Fill and Nakama~\cite{fn2009}; and Vall\'{e}e, Cl\'{e}ment, Fill, and Flajolet~\cite{vcff2009});
and all of them except for Fill and Nakama~\cite{fn2009} and Vall\'{e}e \textit{et al.}~\cite{vcff2009} have conducted the quantification with regard to the number of key comparisons required by the algorithm to achieve its task.  As a result, most of the theoretical results on the complexity of {\tt QuickSelect} are about expectations or distributions for the number of required key comparisons.

However, one can reasonably argue that analyses of {\tt QuickSelect} in terms of the number of key comparisons cannot fully quantify its complexity. For instance, if keys are represented as binary strings, then individual bits of the strings must be compared in order for {\tt QuickSelect} to complete its task, and results obtained by analyzing the algorithm with respect to the number of bit comparisons required to find a target key more accurately reflect actual execution costs. (We will consider bit comparisons as an example of symbol comparisons.)  When {\tt QuickSelect} (or any other algorithm) compares keys during its execution, it must operate on the keys' representations or internal structures, so these should not be ignored in fully characterizing the performance of the algorithm.  Also, symbol-complexity analysis allows us to compare key-based algorithms such as {\tt QuickSelect} and {\tt QuickSort} with digital algorithms such as those utilizing digital search trees.

Fill and Janson \cite{fj2004} pioneered symbol-complexity analysis by analyzing the expected number of bit comparisons required by {\QuickSort}.  They assumed that the algorithm is applied to keys that are i.i.d.\ (independent and identically distributed) from the uniform distribution over $(0, 1)$ and represented (via their binary expansions) as binary strings, and that the algorithm operates on individual bits in order to do comparisons and find the target key.  They found that the expected number of bit comparisons required by {\QuickSort} to sort~$n$ keys is asymptotically equivalent to $n(\ln n)(\lg n)$ (where $\lg$ denotes binary logarithm), whereas the lead-order term of the expected number of \emph{key} comparisons is $2n \ln n$, smaller by a factor of order $\log n$.  In their Section~6 they also considered i.i.d.\ keys drawn from other distributions with density on $(0, 1)$.

By closely following~\cite{fj2004}, Fill and Nakama \cite{fn2009} studied the expected number of bit comparisons required by {\tt QuickSelect}.  More precisely, they treated the case of i.i.d.\ uniform keys represented as binary strings and produced exact expressions for the expected number of bit comparisons by {\tt QuickSelect}$(n, m)$ for general~$n$ and~$m$.  Their asymptotic results were limited to the algorithms {\tt QuickMin}, {\tt QuickMax}, and {\tt QuickRand}.
Here {\tt QuickMin} refers to {\tt QuickSelect} applied to find the smallest key, i.e.,\ to
{\tt QuickSelect}$(n, m)$ with $m = 1$; and {\tt QuickMax} similarly refers to
{\tt QuickSelect}$(n, m)$ with $m = n$.  {\tt QuickRand} is the algorithm that results from taking~$m$ to be uniformly distributed over $\{1, 2, \dots, n\}$.
They showed that the expected number of bit comparisons required by {\tt QuickMin} or
{\tt QuickMax} is asymptotically linear in $n$ with lead-order coefficient approximately equal to $5.27938$.  Thus in these cases the expected number of bit comparisons is asymptotically larger than that of \emph{key} comparisons required to complete the same task only by a constant factor, since the expectation for key comparisons is asymptotically $2 n$.  Fill and Nakama~\cite{fn2009} also found that the expected number of bit comparisons required by {\tt QuickRand} is also asymptotically linear in $n$ (with slope approximately $8.20731$), as for key comparisons (with slope~$3$).

Vall\'{e}e \textit{et al.}\ \cite{vcff2009} extended the average-case analyses of~\cite{fj2004} and~\cite{fn2009} to keys represented by sequences of general symbols generated by any of a wide variety of sources that include memoryless, Markov, and other dynamical sources.  They broadly extended the results of~\cite{fn2009} in another direction as well by treating
{\tt QuickQuant}$(n, \alpha)$ for general $\alpha \in [0, 1]$, not just {\tt QuickMin}, {\tt QuickMax}, and {\tt QuickRand}.  Here the algorithm {\tt QuickQuant}$(n, \alpha)$ (for ``Quick Quantile'') refers to {\tt QuickSelect}$(n, m_n)$ with $m_n / n \to \alpha$.
Roughly summarized, Vall\'{e}e \textit{et al.}\ showed that
if symbols are generated by a suitably nice source, then
the expected number of symbol comparisons in processing a file of~$n$ keys is of order
$n \log^2 n$ for {\tt QuickSort} and, for any $\alpha$, of order~$n$ for
{\tt QuickQuant}$(n, \alpha)$.
(For example, all memoryless sources are suitably nice.)  For a more detailed discussion of sources and the results of Vall\'{e}e \textit{et al.}\ \cite{vcff2009} for {\tt QuickQuant}, see \refS{S:prelims}.

The main purpose of this paper is to extend the average-case analysis of Vall\'{e}e \textit{et al.}\ \cite{vcff2009} by establishing limiting distributions for the number of symbol comparisons.  To our knowledge the present paper is the first to establish a limiting distribution for the number of symbol comparisons required by any key-based algorithm.  Our elementary approach allows us to handle rather general kinds of ``cost'' for comparing two keys, and in particular to recover in a rather direct way known results about the number of key comparisons.  There is no disadvantage to allowing general costs, since our results rely on at most broad limitations on the nature of the cost.
\smallskip

\noindent
{\bf Outline of the paper.\ }We shall be concerned primarily with $\mbox{{\tt QuickQuant}} \equiv \mbox{{\tt QuickQuant}}(n, \alpha)$, which is what we call the algorithm {\tt QuickSelect} when applied to find the key of
rank $m_n$ in a file of size~$n$, where we are given $0 \leq \alpha \leq 1$ and a sequence
$(m_n)$ such that $m_n / n \to \alpha$.  It turns out to be convenient mathematically to analyze a close cousin to {\tt QuickQuant} introduced by Vall\'{e}e \textit{et al.}~\cite{vcff2009}, namely, {\tt QuickVal}, and then treat {\tt QuickQuant} by comparison.  So, after a careful description of the probabilistic models used to govern the generation of keys in \refS{SS:source}, a review of known results about key and symbol comparisons in \refS{SS:known}, and a description of {\tt QuickVal} in~\refS{SS:QQQV}, in \refS{S:QuickVal} we establish limiting-distribution results for {\tt QuickVal} (whose main theorems are \refT{T:convergenceInLp} and \refT{thm:ASConvergence}) and then move on to {\tt QuickQuant} in \refS{S:QQ} (which contains \refT{thm:convergenceOfSQp}, the main theorem of this paper).
\medskip

Subsequent to the research leading to the present paper, and using a rather different approach, the first 
author~\cite{fjfeb2010} has found a limiting distribution for the number of symbol comparisons used by {\tt QuickSort} for a wide variety of probabilistic sources. 

\begr
Although the contraction method has been used in finding limiting distributions for the number of key comparisons required by recursive algorithms such as {\tt QuickSort} (e.g., R\"{o}sler~\cite{r1991}, R\"{o}sler and R\"{u}schendorf~\cite{rr2001}), our analysis does not depend on it. In examining convergence for the number of key comparisons used by {\tt QuickQuant}, Gr\"{u}bel and R\"{o}sler~\cite{gr1996} mentioned that they did not use the contraction method due to the parameter that represents target rank.  (However, they did engage in contraction arguments to characterize the limiting distribution.)  Interestingly, Mahmoud \textit{et al}.~\cite{mms1995} succeeded in establishing a fixed point equation to identify the limiting distributions of the normalized numbers of key comparisons required by {\tt QuickRand}, {\tt QuickMin}, and {\tt QuickMax}.  R\'{e}gnier~\cite{r1989} used martingales to show convergence for the number of key comparisons required by {\tt QuickSort}.
\enr
\newpage

\section{Background and preliminaries}\label{S:prelims}
\subsection{Probabilistic source models for the keys}\label{SS:source}
In this subsection we describe what is meant by a probabilistic source, our model for how the i.i.d.\ keys are generated, using the terminology and notation of Vall\'{e} \textit{et al.}\ \cite{vcff2009}.

Let~$\Sigma$ denote a totally ordered alphabet (set of symbols), assumed to be isomorphic either to $\{0, \dots, r - 1\}$ for some finite~$r$ or to the full set of nonnegative integers, in either case with the natural order; a \emph{word} is then an element of $\Sigma^{\infty}$, i.e.,\ an infinite sequence (or ``string'') of symbols.  We will follow the customary practice of denoting a word
$w = (w_1, w_2, \ldots)$ more simply by $w_1 w_2 \cdots$.

We will use the word ``prefix'' in two closely related ways.  First, the symbol strings belonging to $\Sigma^k$ are called \emph{prefixes} of length $k$, and so $\Sigma^* := \cup_{0 \leq k < \infty} \Sigma^k$ denotes the set of all prefixes of any nonnegative finite length.  Second, if
$w = w_1 w_2 \cdots$ is a word, then we will call
\begin{equation}
\label{prefix}
w(k) := w_1 w_2 \cdots w_k \in \Sigma^k
\end{equation}
its \emph{prefix of length~$k$}.

\emph{Lexicographic order} is the linear order (to be denoted in the strict sense by $\prec$ and in the weak sense by $\preceq$) on the set of words specified by declaring that $w \prec w'$ if (and only if) for some $0 \leq k < \infty$ the prefixes of~$w$ and $w'$ of length~$k$ are equal but $w_{k + 1} < w'_{k + 1}$.  We denote the \emph{cost} of determining $w \prec w'$ when comparing distinct words~$w$ and $w'$ by $c(w, w')$; we will always assume that the
function~$c$ is symmetric and nonnegative.

\begx
\emph{
Here is an example of a natural class of cost functions.
Start with nonnegative symmetric functions $c_i:\Sigma \times \Sigma \to [0, \infty)$,
$i = 1, 2, \dots$, modeling the cost of comparing symbols in the respective $i$th positions of two words.  This allows for the symbol-comparison costs to depend both on the positions of the symbols in the words and on the symbols themselves.  Then, for comparisons of distinct words, define
\begin{equation}
\label{csymb}
c(w, w') := \sum_{i = 1}^{k + 1} c_i(w_i, w'_i) = \sum_{i = 1}^k c_i(w_i, w_i) + c_{k + 1}(w_{k + 1}, w'_{k + 1})
\end{equation}
where~$k$ is the length of the longest common prefix of~$w$ and~$w'$.
}

\emph{
(a)~If $c_i \equiv \delta_{i_0, i}$ (independent of the symbols being compared) for given positive integer $i_0$, then~$c$ is the cost used in counting comparisons of symbols in position $i_0$.  (For example, if $i_0 = 1$ then $c \equiv 1$ is the cost used in counting key comparisons.)  Observe that all finite linear combinations of such cost functions 
$\delta_{i_0, \cdot}$ are of the form~\eqref{csymb}, and therefore, by the Cram\'{e}r--Wold device (\eg,\ \cite[Section~29]{b1995}), if $S_{i_0}$ denotes the total number of comparisons of symbols in position $i_0$, then the \emph{joint} distribution of $(S_1, S_2, \dots)$ can (at least in principle) be obtained by studying cost functions of the form~\eqref{csymb}.
} 

\emph{
(b)~If $c_i \equiv 1$ for all~$i$, then $c \equiv k + 1$ is the cost used in counting symbol comparisons.
}
\enx

A \emph{probabilistic source} is simply a stochastic process $W = W_1 W_2 \cdots$ with
state space~$\Sigma$ (endowed with its total $\sigma$-field) or, equivalently, a random variable~$W$ taking values in $\Sigma^{\infty}$ (with the product $\sigma$-field).  
According to Kolmogorov's consistency criterion
(\eg,~\cite[Theorem~3.3.6]{c2001}), the 
distributions~$\mu$ of such processes are in one-to-one correspondence with consistent specifications of 
finite-dimensional marginals, that is, of the probabilities
$$
p_w := \mu(\{w_1 \cdots w_k\} \times \Sigma^{\infty}), \quad w = w_1 w_2 \cdots w_k \in \Sigma^*.
$$
Here the \emph{fundamental probability} $p_w$ is the probability that a word drawn from~$\mu$ has $w_1 \cdots w_k$ as its length-$k$ prefix.

Because the analysis of {\tt QuickSelect} is significantly more complicated when its input keys are not all distinct, we will restrict attention to probabilistic sources with continuous distributions~$\mu$.  Expressed equivalently in terms of fundamental probabilities, our continuity assumption is  that for any $w = w_1 w_2 \cdots \in \Sigma^{\infty}$ we have $p_{w(k)} \to 0$ as $k \to \infty$, recalling the prefix notation~\eqref{prefix}.

\begx
\label{sourcex}
\emph{We present a few classical examples of sources.  For more examples, and for further discussion, see Section~3 of~\cite{vcff2009}.}
\smallskip

\emph{(a)~In computer science jargon, a \emph{memoryless source} is one with $W_1, W_2, \dots$ i.i.d.\ \ Then the fundamental probabilities $p_w$ have the product form}
$$
p_w = p_{w_1} p_{w_2} \cdots p_{w_k}, \quad w = w_1 w_2 \cdots w_k \in \Sigma^*.
$$

\emph{(b)~A \emph{Markov source} is one for which $W_1 W_2 \cdots$ is a Markov chain.}
\smallskip

\emph{(c)~An intermittent source over the finite alphabet $\Sigma = \{0, \dots, r - 1\}$ models long-range dependence of the symbols within a key and is defined by specifying the conditional distributions
$\Lc(W_j\,|\,W_1, \dots, W_{j - 1})$ in a way that pays special attention to a particular symbol
$\underline \sigma$.  The source is said to be \emph{intermittent  of exponent $\gamma > 0$ with respect 
to~$\underline \sigma$} if $\Lc(W_j\,|\,W_1, \dots, W_{j - 1})$ depends only on the
maximum value~$k$ such that the last~$k$ symbols in the prefix $W_1 \cdots W_{j - 1}$ are all
$\underline \sigma$ and (i)~is the uniform distribution on~$\Sigma$, if $k = 0$; and (ii)~if $1 \leq k \leq j -1$, assigns mass $[k / (k + 1)]^{\gamma}$ to~$\underline \sigma$ and distributes the remaining mass uniformly over the remaining elements of~$\Sigma$.
}
\enx

We next present an equivalent description of probabilistic sources (with a corresponding equivalent condition for continuity) that will prove convenient because it allows us to treat all sources within a uniform framework.  If~$M$ is any measurable mapping from $(0, 1)$ (with its Borel $\sigma$-field) into $\Sigma^{\infty}$ and~$U$ is distributed
unif$(0, 1)$, then $M(U)$ is a probabilistic source.  Conversely, given any probability
measure~$\mu$ on $\Sigma^{\infty}$ there exists a monotone measurable mapping~$M$ such that $M(U)$ has distribution~$\mu$ when $U \sim \mbox{unif}(0, 1)$; here (weakly) \emph{monotone} means that $M(t) \preceq M(u)$ whenever $t \leq u$.  Indeed, if~$F$ is the distribution function
$$
F(w) := \mu\{w' \in \Sigma^{\infty}:w' \preceq w\}, \quad w \in \Sigma^{\infty},
$$
for~$\mu$, then we can always use the inverse probability transform
$$
M(u) := \inf\{w \in \Sigma^{\infty}:u \preceq F(w)\}, \quad u \in (0, 1)
$$
for~$M$.  The measure~$\mu$ is continuous if and only if this~$M$ is strictly monotone.

So henceforth we will assume that our keys are generated as $M(U_1), \dots, M(U_n)$, where $M:(0, 1) \to \Sigma^{\infty}$ is strictly monotone and $U_1, \dots, U_n$ (we will call these the ``seeds'' of the keys) are i.i.d.\ unif$(0, 1)$.  Given a specification of costs $c(w, w')$ in comparing words, we can now define a source-specific notion of cost by setting
$$
\beta(u, t) := c(M(u), M(t)).
$$
In our main application, $\beta_{\mathrm{symb}}(u, t)$ represents the number of symbol comparisons required to compare words with seeds~$u$ and~$t$.

The following associated terminology and notation from~\cite{vcff2009} will also prove useful.  For each prefix $w \in \Sigma^*$, we let $\mathcal{I}_w = (a_w, b_w)$ denote the interval that contains all seeds whose corresponding words begin with~$w$ and $\mu_w := (a_w + b_w) / 2$ its midpoint.  We call $\mathcal{I}_w$ the \emph{fundamental interval} associated with~$w$.
(There  is no need to be fussy as to whether the interval is open or closed or half-open, because the probability that a random seed~$U$ takes any particular value is~$0$.  Also, we always assume that $a_w < b_w$, since the case that $a_w = b_w$ will not concern us.)
The fundamental probability $p_w$ can be expressed as $b_w - a_w$.  The \emph{fundamental triangle} of prefix~$w$, denoted by
$\mathcal{T}_w$, is the triangular region
$$
\mathcal{T}_w := \{(u, t): a_w < u < t < b_w\},
$$
and when~$w$ is the empty prefix we denote this triangle by~$\mathcal{T}$:
$$
\mathcal{T} := \{(u, t): 0 < u < t < 1\}.
$$

For some of our results, the quantity
\begin{equation}
\label{pikdef}
\pi_k := \max\{p_w: w \in \Sigma^k\}
\end{equation}
will play an important role.  
The following definition of a $\Pi$-tame probabilistic source is taken (with slight modification) 
from~\cite{vcff2009}:

\begin{definition}
Let $0 < \gamma < \infty$
and $0 < A < \infty$.  We say that the source is \emph{$\Pi$-tame (with 
parameters~$\gamma$ and~$A$)} if the sequence $(\pi_k)$ at~\eqref{pikdef} satisfies
$$
\pi_k \leq A (k + 1)^{- \gamma}\mathrm{\ for\ every\ }k \geq 0.
$$
\end{definition}

Observe that a $\Pi$-tame source is always continuous.  There is a related condition for cost functions~$\beta$ that will be assumed (for suitable values of the parameters) in some of our results:

\begin{definition}
\label{etame}
Let $0 < \epsilon < \infty$ and $0 < c < \infty$.  We say that the symmetric cost function $\beta \geq 0$ is \emph{tame (with parameters~$\epsilon$ and~$c$)} if
$$
\beta(u, t) \leq c (t - u)^{- \epsilon}\mathrm{\ for\ all\ }(u, t) \in T.
$$
We say that~$\beta$ is \emph{$\epsilon$-tame} if it is tame with parameters~$\epsilon$ and~$c$ for some~$c$. 
\end{definition}

We leave it to the reader to make the simple verification that a source is $\Pi$-tame with parameters~$\gamma$ 
and~$A$ if and only if $\beta_{\mathrm{symb}}$ is tame with parameters $\epsilon = 1 / \gamma$ and $c = A^{1 / \gamma}$.

\begr
\label{R:g-tame}
(a)~Many common sources have geometric decrease in $\pi_k$ (call these ``g-tame'') and so for \emph{any} 
$\gamma$ are $\Pi$-tame with parameters~$\gamma$ and~$A$ for suitably chosen $A \equiv A_{\gamma}$ [equivalently, the symbol-comparisons
cost~$\beta_{\mathrm{symb}}$ is $\epsilon$-tame for any~$\epsilon$; in fact, if $\pi_k \leq b^{- k}$ for every~$k$, then
\beqnn
\beta_{\mathrm{symb}}(u,t) \leq 1 + \log_{b}\frac{1}{t-u}\mathrm{\ for\ all\ }(u, t) \in \mathcal{T}].
\eeqnn

For example, a memoryless source satisfies $\pi_k = p^k_{\max}$, where
$$
p_{\max} := \sup_{w \in \Sigma^1} p_w
$$
satisfies $p_{\max} < 1$ except in the highly degenerate case of an essentially single-symbol alphabet.  We also have $\pi_k \leq p^k_{\max}$ for any Markov source, where now $p_{\max}$ is the supremum of all one-step transition probabilities, and so such a source is g-tame provided $p_{\max} < 1$.  Expanding dynamical sources 
(\cf\ Cl\'{e}ment, Flajolet, and Vall\'{e}e~\cite{cfv2001})
are also g-tame.
\smallskip

(b)~For an intermittent source as in Example~\ref{sourcex}, for all large~$k$ the maximum probability $\pi_k$ is attained by the word $\underline \sigma^k$ and equals
$$
\pi_k = r^{-1} k^{- \gamma}.
$$
Intermittent sources are therefore examples of $\Pi$-tame sources for which $\pi_k$ decays at a truly inverse-polynomial rate, not an exponential rate as in the case of g-tame sources.
\enr

\subsection{Known results for the numbers of key and symbol comparisons}\label{SS:known}
In this subsection we give for {\tt QuickSelect} an abbreviated review of what is already known about the distribution of the number of key comparisons ($\beta \equiv 1$ in our notation) and (from Vall\'{e}e\ \textit{et al.}\ \cite{vcff2009}) about the expected number of symbol comparisons ($\beta = \beta_{\mathrm{symb}}$).  To our knowledge, no other cost functions have previously been considered, nor has there been any treatment of the full distribution of the number of symbol comparisons.

Let $K_{n,m}$ denote the number of key comparisons required by the algorithm to find a key of rank $m$ in a file of $n$ keys (with $1\leq m\leq n$).  Thus $K_{n,1}$ and $K_{n,n}$ represent the key comparison costs required by {\tt QuickMin} and {\tt QuickMax}, respectively.
(Clearly $K_{n,1}\Leq K_{n,n}$).  It has been shown (see Mahmoud \textit{et al}.~\cite{mms1995}, Hwang and Tsai~\cite{ht2002}) that as $n \rightarrow \infty$, $K_{n,1} / n$ converges in law to the Dickman distribution, which can be described as the distribution of the perpetuity
\beqnn
1 + \sum_{k\geq 1} U_1\cdots U_k,
\eeqnn
where $U_k$ are i.i.d.\ uniform$(0, 1)$.  Mahmoud \textit{et al}.~\cite{mms1995}
established a fixed-point equation for the limiting distribution of the normalized (by dividing by~$n$) number of key comparisons required by {\tt QuickRand} and also explicitly identified this limiting distribution.

By using process-convergence techniques, Gr\"{u}bel and R\"{o}sler~\cite[Theorem~8]{gr1996} identified, for each $0 \leq \alpha < 1$, a nondegenerate random variable $K(\alpha)$ to which $K_{n, \lf \alpha n \rf + 1} / n$ converges in distribution; see also the fixed-point equation in their Theorem~10, and Gr\"{u}bel~\cite{g1998}, who used a Markov chain approach and characterized the limiting distribution in his Theorem~3.  Earlier, Devroye~\cite{d1984} had shown that
$$
\sup_{n \geq 1}\,\max_{1 \leq m \leq n} \PP(K_{n, m} \geq t n) \leq C \rho^t
$$
for any $\rho > 3/4$ and some $C \equiv C(\rho)$.

Concerning moments, Gr\"{u}bel and R\"{o}sler~\cite[Theorem~11]{gr1996} showed that
$\EE\,K(\alpha) = 2 [1 - \alpha \ln \alpha - (1 - \alpha) \ln(1 - \alpha)]$ and Paulsen~\cite{pa1995} calculated higher-order moments of $K(\alpha)$.  Gr\"{u}bel~\cite[end of Section~2]{g1998} proved convergence of the moments for finite~$n$ to the corresponding moments of the limiting $K(\alpha)$.

Prior to the present paper, only expectations have been studied for the number of symbol comparisons for {\tt QuickQuant}.  The current state of knowledge is summarized by part~(i) of Theorem~2 in Vall\'{e}e \textit{et al.}~\cite{vcff2009} (see also their accompanying Figures 1--3); we refer the reader to~\cite{vcff2009} for the other parts of the theorem, which routinely specialize part~(i) to {\tt QuickMin}, {\tt QuickMax}, and {\tt QuickRand}.

To review their result we need the notation and terminology of \refS{SS:source} and a bit more.
Using the non-standard abbreviations $y^+ := (1/2) + y$ and $y^- := (1/2) - y$ and the convention $0 \ln 0 := 0$, we define
$$
H(y) :=
\begin{cases}
- (y^+ \ln y^+ + y^- \ln y^-), & \mathrm{if\ }0 \leq y \leq 1/2 \\
y^- (\ln y^+ - \ln | y^- |), & \mathrm{if\ }y \geq 1/2
\end{cases}
$$
and then set $L(y) := 2 [1 + H(y)]$.
According to Theorem~2(i) in~\cite{vcff2009}, for any $\Pi$-tame source the mean number of symbol comparisons for {\tt QuickQuant}$(n, \alpha)$ is asymptotically $\rho\,n + O(n^{1 - \delta})$ for some $\delta > 0$.  Here $\rho \equiv \rho(\alpha)$ and~$\delta$ both
depend on the probabilistic source, with
\begin{equation}
\label{rhodef}
\rho := \sum_{w \in \Sigma^*} p_w L\left( \left| \frac{\alpha - \mu_w}{p_w} \right| \right).
\end{equation}
They derive~\eqref{rhodef} by first proving 
the equality
\begin{equation}
\label{rhoint}
\rho = \int_{\cal T} \beta(u, t) \left[ (\alpha \vee t) - (\alpha \wedge u) \right]^{-1}\,du\,dt
\end{equation}
for $\Pi$-tame sources with $\gamma > 1$.

\subsection{{\tt QuickQuant} and {\tt QuickVal}}\label{SS:QQQV}
Let $\SQ \equiv \SQ(\alpha)$ denote the total cost required by {\tt QuickQuant}$(n, \alpha)$.
To prove convergence of $\SQ / n$ (in suitable senses to be made precise later), we exploit an idea introduced by Vall\'{e}e \textit{et al.}\ \cite{vcff2009} and begin with the study of a related algorithm, called ${\mathrm {\tt QuickVal}} \equiv \mathrm{{\tt QuickVal}}(n, \alpha)$, which we now describe.  {\tt QuickVal} is admittedly somewhat artificial and inefficient; it is important to keep in mind that we study it mainly as an aid to studying {\tt QuickQuant}.

Having generated $n$ seeds and then $n$ keys $M_1, \dots, M_n$ (say) using our probabilistic source, {\tt QuickVal} is a recursive randomized algorithm to find the rank of the additional word $M(\alpha)$ in the set $\{M_1, \dots, M_n, M(\alpha)\}$; thus, while {\tt QuickQuant} finds the value of the $\alpha$-quantile in the \emph{sample} of keys, {\tt QuickVal} dually finds the rank of the \emph{population} $\alpha$-quantile in the augmented set.  First, {\tt QuickVal} selects a pivot uniformly at random from the set of keys $\{M_1, \dots, M_n\}$ and finds the rank of the pivot by (a)~comparing the pivot with each of the other keys (we will count these comparisons) and (b)~comparing the pivot with $M(\alpha)$ (we will find it convenient not to count the cost of this comparison in the total cost).  With probability one, the pivot key will differ from the word $M(\alpha)$.  If $M(\alpha)$ is smaller than the pivot key, then the algorithm operates recursively on the set of keys smaller than the pivot and determines the rank of the word $M(\alpha)$ in the set $\mathcal{M}_{\mathrm{smaller}} \cup \{M(\alpha)\}$, where $\mathcal{M}_{\mathrm{smaller}}$ denotes the set of keys smaller than the pivot. Similarly, if $M(\alpha)$ is greater than the pivot key, then the algorithm operates recursively on the set of keys larger than the pivot [together with the word $M(\alpha)$].  Eventually the set of words on which the algorithm operates reduces to the singleton $\{M(\alpha)\}$, and the algorithm terminates.

Notice that the operation of {\tt QuickVal} is quite close to that of
{\tt QuickQuant}, for the same value of~$\alpha$; we expect running costs of the two algorithms to be close, since when~$n$ is large the rank of $M(\alpha)$ in $\{M_1, \dots, M_n, M(\alpha)\}$ should be close (in relative error terms) to $\alpha n$.  In fact, we will show that if $\SV \equiv \SV(\alpha)$ denotes the total cost of executing {\tt QuickVal}$(n, \alpha)$, then $\SQ / n$ and $\SV / n$ have the same limiting distribution, assuming only that the cost function~$\beta$ is $\epsilon$-tame for suitably small~$\epsilon$.  In fact, we will show that when all the random variables $S^{\mathrm{Q}}_1, S^{\mathrm{Q}}_2, \dots$ and $S^{\mathrm{V}}_1, S^{\mathrm{V}}_2, \dots$ are strategically defined on a common probability space, then $\SQ / n$ and $\SV / n$ both converge in $L^p$ to a common limit for $1 \leq p < \infty$.

\section{Analysis of {\tt QuickVal}}\label{S:QuickVal}
Following some preliminaries in \refS{SS:QValpreliminaries}, in
Section~\ref{SS:QValConvergenceInLp} we show that for $1 \leq p < \infty$, a suitably defined $\SV / n$ converges in $L^p$ to a certain random variable~$S$ (defined at the end of Section~\ref{SS:QValpreliminaries}) provided only that $\EE\,S < \infty$.  We also show that, when the cost function is suitably tame, $\SV / n$ converges almost surely to $S$; see Theorem~\ref{thm:ASConvergence} in Section~\ref{SS:QValASConvergence}.  We derive an integral expression for $\EE\,S$ valid for a completely general cost function~$\beta$ in
Section~\ref{SS:QValIntegralExpression} and use it to compute the expectation when
 $\beta \equiv 1$.  In Section~\ref{SS:QValSeriesExpression}, we focus on $\EE\,S$ with $\beta = \beta_{\mathrm{symb}}$ and derive a series expression for the expectation.  Few comparisons of results obtained here with the known results reviewed in \refS{SS:known} are made in the present section; most such comparisons are deferred to (the first paragraph of) \refS{S:QQ}, where the previously-studied algorithm of greater interest, {\tt QuickQuant}, is treated.

\subsection{Preliminaries}\label{SS:QValpreliminaries}
Our goal is to establish a limit, in various senses, for the ratio of the total cost
required by {\tt QuickVal} when applied to a file of $n$ keys to~$n$.  It will be both natural and convenient to define all these total costs, one for each value of~$n$, in terms of a single infinite sequence $(U_i)_{i \geq 1}$ of seeds that are i.i.d.\ uniform(0,~1).  Indeed, let $L_0 := 0$ and $R_0:= 1$.  For $k \geq 1$, inductively define
\begin{eqnarray}
\tau_k &:=& \inf\{i:\, L_{k-1} < U_i < R_{k-1}\},\label{pivotIndex} \\
L_k &:=& {\bf 1}(U_{\tau_k} < \alpha) U_{\tau_{k}} + {\bf 1}(U_{\tau_k} > \alpha) L_{k-1},\label{lowerBound} \\
R_k &:=& {\bf 1}(U_{\tau_k} < \alpha) R_{k-1} + {\bf 1}(U_{\tau_k} > \alpha) U_{\tau_{k}},\label{upperBound} \\
S_{n, k} &:=& \sum_{i:\,\tau_k < i \leq n} {\bf 1}(L_{k-1} < U_i < R_{k - 1})\,\beta(U_i, U_{\tau_k}).\label{Snk}
\end{eqnarray}
(Note that $S_{n, k}$ vanishes if $\tau_k \geq n$.)
We then claim that, for each~$n$,
\begin{equation}
\label{SV}
\SV := \sum_{k \geq 1} S_{n, k}
\end{equation}
has the distribution of the total cost required by {\tt QuickVal}$(n, \alpha)$.

We offer some explanation here.  For each $k \geq 1$, the random interval $(L_{k - 1}, R_{k - 1})$ (whose length decreases monotonically in~$k$) contains both the target seed~$\alpha$ and the seed $U_{\tau_k}$ corresponding to the $k$th pivot; the interval contains precisely those seed values still under consideration after $k - 1$ pivots have been performed.  The only difference between how we have defined~$\SV$ and how it is usually defined is that we have chosen the initial pivot seed to be the \emph{first} seed rather than a random one, and have made this same change recursively.  But our change is permissible because of the following basic probabilistic fact: If $U_1, \dots, U_N, M$ are independent random variables with $U_1, \dots, U_N$ i.i.d.\ uniform$(0, 1)$ and~$M$ uniformly distributed on $\{1, \dots, N\}$, then $U_M$, like $U_1$, is distributed uniform$(0, 1)$.  Thus the conditional distribution of $U_{\tau_k}$ given $(L_{k - 1}, R_{k - 1})$ is uniform$(L_{k - 1}, R_{k - 1})$.

We illustrate our notation for the first two pivots.  First, $\tau_1 = 1$; that is, the seed of the first pivot is the uniform$(0, 1)$ random variable $U_1$.  After that, if $\alpha < U_1$ then the seed $U_{\tau_2}$ of the second pivot is chosen as the first seed falling in $(0, U_1)$, while if $\alpha > U_1$ then $U_{\tau_2}$ is the first seed falling in $(U_1, 1)$.  We note that if
$\alpha = 0$ (which means that we are dealing with the total cost required by {\tt QuickMin}), then the first of these two cases is always the one that applies and so for every $k \geq 1$ we have $L_k = 0$ and $R_k = U_{\tau_k}$; we then have that $U_{\tau_k}$ is just the $k$th record low value among $U_1, U_2, \dots$.

In order to describe the limit of $\SV / n$, we let
\begin{eqnarray}
I(t, x, y) &:=& \int_x^y\!\beta(u, t)\,du,\nonumber\\
I_k &:=& I(U_{\tau_k}, L_{k - 1}, R_{k-1}),\label{I_k}\\
S &:=& \sum_{k \geq 1} I_k.\label{S}
\end{eqnarray}
Notice that in the case $\beta \equiv 1$ of key comparisons we have $I(t, x, y) \equiv y - x$ and so $I_k = R_{k - 1} - L_{k - 1}$.

In Section~\ref{SS:QValConvergenceInLp} we show for $1 \leq p < \infty$ that $\SV / n$ converges in $L^p$ to~$S$ as $n \rightarrow \infty$ under proper technical conditions.  Under a stronger assumption, we will also prove almost sure convergence in Section~\ref{SS:QValASConvergence}.

\subsection{Convergence of $\SV / n$ in $L^p$ for $1 \leq p < \infty$}\label{SS:QValConvergenceInLp}
Theorem~\ref{T:convergenceInLp} is our main result concerning {\tt QuickVal}.  To state the result, we need the following notation, extending that of~\eqref{I_k}:
\begin{eqnarray}
I_p(t, x, y) &:=& \int_x^y\!\beta^p(u, t)\,du,\nonumber\\
I_{p, k}&:=& I_p(U_{\tau_k}, L_{k - 1}, R_{k-1}),\label{Ipk}.
\end{eqnarray}

\begin{theorem}\label{T:convergenceInLp}
If $1 \leq p < \infty$ and
\begin{equation}
\label{techp}
\sum_{k \geq 1} (\EE\,I_{p, k})^{1 / p} < \infty,
\end{equation}
then $\SV / n$ converges in $L^p$ (and therefore also in probability and in distribution)
to~$S$ as $n \to \infty$.
\end{theorem}

\begr\label{rmk:L1}
For $p=1$, notice that the assumption of Theorem~\ref{T:convergenceInLp} only requires that $\EE\,S < \infty$, which is equivalent to the assertion that
$\sum_{k \geq 1} \EE\,I_k < \infty.$
\enr

\begin{proof}
We use $\| \cdot \|$ to denote $L^p$-norm.  
We will utilize the  
$L^p$ law of large numbers ($L^p$LLN),
which
asserts that for $1 \leq p < \infty$ and i.i.d.\ random variables $\xi_1, \xi_2, \dots$ with finite $L^p$-norm, the sample means $\bar{\xi}_n = n^{-1} \sum_{i = 1}^n \xi_i$ converge in $L^p$ to the expectation.  Because the $L^p$LLN is not as well known as the strong law of large numbers, we provide a proof.  We may assume with no loss of generality that $\EE\,\xi_1 = 0$.  Let $Z_{- n} := \bar{\xi}_n$ for $n = 1, 2, \dots$; then $Z$ is a martingale (see, \eg,\ \cite[proof of Theorem~9.5.6]{c2001}), and therefore the process $(|Z_n|)^p)_{n = \dots, -2, -1}$ is a nonnegative submartingale.  It therefore follows 
\cite[Theorem~9.4.7 (d) $\Rightarrow$ (b)]{c2001} that 
$|\bar{\xi}_n|^p$ converges in $L^1$ to~$0$ as $n \to \infty$, as desired. 

Returning to the setting of the theorem, fix~$k$.  Conditionally
given the quadruple $C_k = (L_{k-1}, R_{k-1}, \tau_k, U_{\tau_k})$, the random variables $U_i$ with $i > \tau_k$ are i.i.d.\ uniform$(0, 1)$.  By the $L^p$LLN we
have [using the convention $0 / 0 = 0$ for $S_{n, k} / (n - \tau_k)$ when $n = \tau_k$]
\begin{equation}
\label{condLp}
\EE\left[ \left. \left| \frac{S_{n, k}}{n - \tau_k} - I_k \right|^p\,\right|\,C_k \right]\,\asto\,0\mbox{\ \ as $n \to \infty$}
\end{equation}
since, with~$U$ uniformly distributed and independent of all the $U_i$'s,
\begin{equation}
\label{IpkE}
\EE[{\bf 1}(L_{k-1} < U < R_{k-1})\,\beta(U, U_{\tau_k})\,|\,C_k]
= I_k.
\end{equation}
For our conditional application of the $L^p$LLN in~\eqref{condLp}, it is sufficient to assume only that the probabilistic source and the cost function $\beta \geq 0$ are such that $I_{p, k}$ is
a.s.\ finite, and this clearly holds by~\eqref{techp}.

Our next goal is to show that the left side
of~\eqref{condLp} is dominated by a single random variable (depending on the fixed value
of~$k$) with finite expectation, and then we will apply the dominated convergence theorem.
For every~$n$, using the convexity of $x^p$ for $x > 0$ we obtain
$$
\EE\left[ \left. \left| \frac{S_{n, k}}{n - \tau_k} - I_k \right|^p\,\right|\,C_k \right]
\leq 2^{p - 1} \left( \EE\left[ \left. \left( \frac{S_{n, k}}{n - \tau_k} \right)^p\,\right|\,C_k \right]  + I^p_k \right).
$$
We claim that each of the two terms multiplying $2^{p - 1}$ on the right here is bounded by $I_{p, k}$.
First, using the triangle inequality for conditional $L^p$-norm given $C_k$, the fact that the random variables summed to obtain $S_{n, k}$ are conditionally i.i.d.\ given $C_k$, and the
definition~\eqref{Ipk} of $I_{p, k}$, we can bound the $p$th root of the first term by
\begin{eqnarray}
\lefteqn{\hspace{-.2in}\left\{ \EE\left[ \left. \left( \frac{S_{n, k}}{n - \tau_k} \right)^p\,\right|\,C_k \right]
\right\}^{1 / p}} \nonumber \\
&\leq& \frac{1}{n - \tau_k} \sum_{i:\tau_k < i \leq n} \left\{ \EE\left[ {\bf 1}(L_{k - 1} < U_i < R_{k - 1}) \beta^p(U_i, U_{\tau_k})\,|\,C_k \right] \right\}^{1 / p} \nonumber \\
&=& \left\{ \EE\left[ {\bf 1}(L_{k - 1} < U < R_{k - 1}) \beta^p(U, U_{\tau_k})\,|\,C_k \right]
\right\}^{1 / p} = I_{p, k}^{1 / p} \label{recall}
\end{eqnarray}
with~$U$ as at~\eqref{IpkE}.   For the second term we observe that
$[I_k / (R_{k - 1} - L_{k - 1})]^p$ is the $p$th power of the absolute value of a uniform average and so is bounded by the corresponding uniform average of absolute values of $p$th powers, namely, $I_{p, k} / (R_{k - 1} - L_{k - 1})$; thus
\begin{equation}
\label{Ipkineq}
I^p_k \leq (R_{k - 1} - L_{k - 1})^{p - 1} I_{p, k} \leq I_{p, k}.
\end{equation}
So we conclude that
$$
\EE\left[ \left. \left| \frac{S_{n, k}}{n - \tau_k} - I_k \right|^p\,\right|\,C_k \right] \leq 2^p I_{p, k}.
$$
Thus it follows from $\EE\,I_{p, k} < \infty$ [which follows from \eqref{techp}] and the dominated convergence theorem that
\begin{equation}
\label{Lptilde}
\EE \left| \frac{S_{n, k}}{n - \tau_k} - I_k \right|^p \to 0\mbox{\ \ as $n \to \infty$}.
\end{equation}

Next, we will show from~\eqref{Lptilde} that, for each $k$,
\begin{equation}
\label{Lp}
\EE \left| \frac{S_{n, k}}{n} - I_k \right|^p \to 0\mbox{\ \ as $n \to \infty$}
\end{equation}
by proving that
$$
d_{n, k} \equiv d_{p, n, k} := \EE \left| \frac{S_{n, k}}{n} - \frac{S_{n, k}}{n - \tau_k} \right|^p = \EE \left( \frac{\tau_k}{n} \frac{S_{n, k}}{n - \tau_k} \right)^p
$$
vanishes in the limit as $n \to \infty$.  Indeed, the corresponding conditional expectation given $C_k$ is
$$
{\bf 1}(\tau_k < n) \left( \frac{\tau_k}{n} \right)^p \EE\left[ \left. \left( \frac{S_{n, k}}{n - \tau_k} \right)^p\,\right|\,C_k \right] \leq {\bf 1}(\tau_k < n) \left( \frac{\tau_k}{n} \right)^p I_{p, k}
$$
recalling the inequality~\eqref{recall}.
So again using $\EE\,I_{p, k} < \infty$ and
applying the dominated convergence theorem we find that $d_{n, k} \to 0$, as desired.

Finally, we show that $\SV / n$ converges to~$S$ in $L^p$. Since we have termwise
\mbox{$L^p$-convergence} of
$\SV / n$ to~$S$ by~\eqref{Lp}, the triangle inequality for $L^p$-norm and the dominated convergence theorem for sums imply that
$\SV / n$ converges in $L^p$ to~$S$ provided we can find a summable sequence $b_k$ such that
$$
\max \left\{ \sup_{n \geq 1} \left\| \frac{S_{n, k}}{n} \right\|_p, \left\|I_k\right\|_p \right\} \leq b_k.
$$
But, for any $n \geq 1$, we have [by taking $p$th powers in~\eqref{recall}, then taking expectations, then taking $p$th roots]
$$
\left\| \frac{S_{n, k}}{n} \right\|_p \leq \left\| \frac{S_{n, k}}{n - \tau_k} \right\|_p \leq
(\EE\,I_{p, k})^{1 / p}.
$$
Further, $\| I_k \|_p \leq (\EE\,I_{p, k})^{1 / p}$ follows from~\eqref{Ipkineq}.
Finally, $b_k := (\EE\,I_{p, k})^{1 / p}$ is assumed to be summable.  Thus $\SV / n$ converges to~$S$ in~$L^p$.
\end{proof}

\begr
\label{rmk:keys}
Letting $K_n$ denote the number of key comparisons required by {\tt QuickVal}$(n, \alpha)$, we find from Theorem~\ref{T:convergenceInLp} with $\beta \equiv 1$ that $K_n / n$ converges in $L^p$ ($1 \leq p < \infty$) to
$$
K := \sum_{k = 0}^{\infty} (R_k - L_k).
$$
(In Section~\ref{SS:QValIntegralExpression}, we will explicitly show the required condition
that $\EE\, K < \infty$; see Remark~\ref{rmk:finiteESforKeys}.)

Suppose $\alpha = 0$; then
the number of key comparisons $K_n$ for {\tt QuickVal}$(n, \alpha)$ is the same as for
{\tt QuickMin}.
In this case Theorem~\ref{T:convergenceInLp} gives
\begin{equation}
\label{keys}
\frac{K_n}{n}\,\Lpto\,K = 1 + \sum_{k \geq 1} U_{\tau_k}
\end{equation}
for $1 \leq p < \infty$.  The limiting random variable~$K$ has mean~$2$ and the same so-called Dickman distribution as the perpetuity
\begin{equation}
\label{perp}
1 + \sum_{k \geq 1}^{\infty} U_1 \cdots U_k.
\end{equation}
That~\eqref{keys}--\eqref{perp} holds is well known (e.g.,\ Mahmoud \textit{et al}.~\cite{mms1995}, Hwang and Tsai~\cite{ht2002}).
\enr

\subsection{Almost Sure Convergence of $\SV / n$}\label{SS:QValASConvergence}
Under a tameness assumption, we can also show that $\SV / n$ converges to $S$ almost surely.
(Recall Definition~\ref{etame}.)

\begin{theorem}\label{thm:ASConvergence}
Suppose that the cost~$\beta$ is $\epsilon$-tame for some $\epsilon < 1/4$. Then $\SV / n$ defined at~\eqref{SV} converges to $S$ almost surely.
\end{theorem}

Before proving this theorem, we establish three lemmas bounding various quantities of interest.

\begin{lemma}\label{boundForEdiffp}
For any $p > 0$ and $k\geq 1$, we have
\beqnn
\EE (R_{k}-L_{k})^p \leq \left(\frac{2-2^{-p}}{p+1}\right)^k.
\eeqnn
\end{lemma}
\noindent Here note that for all $p > 0$ we have
\begin{equation}
\label{commonratio}
0 < \frac{2-2^{-p}}{p+1} < 1.
\end{equation}

\begin{proof}
Fix $p > 0$ and $k \geq 1$.  Since $R_0 - L_0 = 1$, it is sufficient to prove that
\beqnn
\EE [(R_{k}-L_{k})^p | L_{k-1}, R_{k-1}] \leq \frac{2-2^{-p}}{p+1}(R_{k-1}-L_{k-1})^p.
\eeqnn
Condition on $(L_{k - 1}, R_{k - 1})$; then with~$U$ uniformly distributed over
$(L_{k - 1}, R_{k - 1})$ we have the stochastic inequality
\beqnn
R_k - L_k \leq_{\mbox{\scriptsize st}} \max\{U - L_{k - 1}, R_{k - 1} - U\}.
\eeqnn
Thus for $L_{k - 1} \neq R_{k - 1}$, with
$$
A_{k - 1} := (L_{k - 1} + R_{k - 1}) / 2,
$$
we have
\begin{eqnarray*}
\lefteqn{\EE [(R_k -L_k)^p\,|\,L_{k-1}, R_{k-1}]} \\
&\leq& \EE [(\max\{U - L_{k - 1}, R_{k - 1} - U\})^p\,|\,L_{k - 1}, R_{k - 1}] \\
&=& (R_{k - 1} - L_{k - 1})^{-1} \left[ \int_{L_{k - 1}}^{A_{k - 1}}\!(R_{k - 1} - u)^p\,du + \int_{A_{k - 1}}^{R_{k - 1}}\!(u - L_{k - 1})^p \right]\,du \\
&=& \frac{2-2^{-p}}{p+1}(R_{k-1}-L_{k-1})^p,
\end{eqnarray*}
as desired.
\end{proof}

\begl
\label{L:intbound}
Suppose that the cost~$\beta$ is tame with parameters~$\epsilon$ and~$c$.  Then for any interval
$(a, b) \subseteq (0, 1)$, any $t \in (a, b)$, and any $0 \leq q < 1 / \epsilon$,  we have
$$
\int_a^b\!\beta^q(u, t)\,du  \leq  \frac {2^{q \epsilon} c^q} {1- q \epsilon} (b - a)^{1-q \epsilon}.
$$
\enl

\begin{proof}
Using the tameness assumption, integration immediately gives
$$
\int_a^b\!\beta^q(u, t)\,du  \leq  \frac {c^q} {1- q \epsilon} \left[ (t - a)^{1-q \epsilon}
+ (b - t)^{1-q \epsilon} \right].
$$
The lemma now follows from the concavity of $x^{1-q \epsilon}$ for $x > 0$.
\end{proof}

The next lemma is a simple consequence of the preceding two.

\begl
\label{L:Imomentbound}
Suppose that the cost~$\beta$ is tame with parameters $\epsilon < 1$ 
and~$c$.
Then for any $k \geq 1$ and any  $q > 0$, we have
$$
\EE\, I_k ^q \leq \left( \frac {2^{ \epsilon} c} {1-  \epsilon} \right)^q \left( \frac{2 - 2^{ - q (1-\epsilon)}} {q (1 - \epsilon) + 1} \right)^{k - 1},
$$
and so $\sum_k \EE\,I^q_k < \infty$ geometrically quickly.
\enl

\begin{proof} Recalling
$$
I_k = \int_{L_{k-1}}^{R_{k-1}}\!\beta(u,U_{\tau_k})\,du,
$$
we find from \refL{L:intbound} that
$$
I_k \leq \left( \frac{2^{\epsilon} c}{1-\epsilon} \right)  (R_{k-1}-L_{k-1})^{1 - \epsilon}.
$$
By application of \refL{boundForEdiffp} we thus obtain the desired bound on $\EE\,I^q_k$.  The series-convergence assertion follows from the observation~\eqref{commonratio}.
\end{proof}

Now we prove Theorem~\ref{thm:ASConvergence}.
\begin{proof}[Proof of Theorem~\ref{thm:ASConvergence}]
Clearly
it suffices to show that
\begin{equation}
\frac{\SV}{n} - \frac{\St_n}{n}\ \asto\ 0\label{QValASconvergenceP1}
\end{equation}
and
\begin{equation}
\frac{\St_n}{n} - S\ \asto\ 0\label{QValASconvergenceP2},
\end{equation}
where
$$
\St_n := \sum_{k\geq 1}(n-\tau_k)^{+} I_k.
$$
We tackle~\eqref{QValASconvergenceP2} first and then~\eqref{QValASconvergenceP1}.

By the monotone convergence theorem, $\St_n / n \uparrow S$ almost surely.  But
from \refL{L:Imomentbound} (using only $\epsilon < 1$) we have $\EE\,S = \sum_{k \geq 1} \EE\, I_k < \infty$, which implies that $S < \infty$ almost surely.  Hence~\eqref{QValASconvergenceP2} follows.

Our proof of~\eqref{QValASconvergenceP1} both is inspired by and follows along the same lines as the ``fourth-moment proof'' of the strong law of large numbers described in
Ross~\cite[Chapter~8]{r2002}; as in that proof, we prefer easy calculations involving fourth moments to more difficult ones involving tail probabilities---perhaps with the expense that the value $1 / 4$ in the statement of \refT{thm:ASConvergence} could be raised by more sophisticated arguments.
For~\eqref{QValASconvergenceP1} it suffices to show that, for any $\delta > 0$,
$$
\PP\left( \left|\frac{\SV}{n} - \frac{\St_n}{n} \right| > \delta \mbox{\ i.o.} \right) = 0,
$$
for which it is sufficient by the first Borel--Cantelli lemma and Markov's inequality to show that
\beqn
\sum_{n\geq 1}\EE \left(\frac{\SV}{n} - \frac{\St_n}{n} \right)^4 < \infty.\label{showThisFor4thM}
\eeqn
Here, by the triangle inequality for the $L^4$ norm,
\beqn
\sum_{n\geq 1}\EE \left(\frac{\SV}{n} - \frac{\St_n}{n} \right)^4
&\leq& \sum_{n\geq 1} \left[\sum_{k\geq 1} \left\| \frac{S_{n,k}}{n}-
\frac{(n-\tau_k)^+}{n}I_k\right\|_4\right]^4\nonumber\\
&=& \sum_{n\geq 1} \left[\sum_{k\geq 1} \left\| \frac{(n-\tau_k)^+}{n}\left(\frac{S_{n,k}}{n-\tau_k}-I_k\right)\right\|_4\right]^4,\label{boundFor4thM}
\eeqn
where we again use the convention $0 / 0 = 0$ for $S_{n, k} / (n - \tau_k)$ when $n = \tau_k$. As in the proof of Theorem~\ref{T:convergenceInLp}, we let $C_k$ denote the quadruple $(L_{k-1}, R_{k-1}, \tau_k, U_{\tau_k})$.  Also we define
\beqnn
\tilde{I}_k := {\bf 1}(L_{k-1}<U<R_{k-1})\beta(U, U_{\tau_k}).
\eeqnn
and
\beqnn
M_m(k):=\EE [(\tilde{I}_k-I_k)^m |C_k],
\eeqnn
where $U$ is unif$(0, 1)$ and independent of $C_k$.  Then routine calculation (see
Ross~\cite[Section~8.4]{r2002}) shows that
\beqn
& &\EE\left[\frac{(n-\tau_k)^+}{n}\left(\frac{S_{n,k}}{n-\tau_k}-I_k\right)\right]^4 = \EE\left[\EE \left[\left.\left\{\frac{(n-\tau_k)^+}{n}\left(\frac{S_{n,k}}{n-\tau_k}-I_k\right)\right\}^4 \right| C_k\right]\right]\nonumber\\
& & {} = \EE\left\{\left[\frac{(n-\tau_k)^+}{n}\right]^4 \left[\frac{(n-\tau_k)^+ M_4(k) + 3 (n-\tau_k)^+ (n-\tau_k-1)^+ M_2^2(k)}{[(n-\tau_k)^+]^4}\right]\right\}\nonumber\\
& & {} \leq \EE\left\{n^{-4} \left[n M_4(k) + 3n(n-1) M_4(k)\right]\right\} \leq 3 n^{-2}\, \EE\,M_4(k),\nonumber\\ \label{forBoundFor4thM}
\eeqn
where the first inequality holds because $M_4(k) \geq M_2^2(k)$.

We will show that $\EE\,M_4(k)$ decays geometrically and then use that fact to
prove~\eqref{showThisFor4thM}.  Since $(a-b)^4 \leq 8(a^4 + b^4)$ for any real $a$ and $b$, we have
\beqn
M_4(k) &\leq& 8 \left( \EE [\tilde{I}_k^4|C_k] + I_k^4 \right).\label{EM4}
\eeqn
First, using \refL{L:Imomentbound} we find (using only $\epsilon < 1$) that $\EE\,I^4_k < \infty$ decays geometrically:
\beqn
\EE\, I_k ^4 \leq \left( \frac {2^{ \epsilon} c} {1-  \epsilon} \right)^4 \left( \frac{2 - 2^{ - 4 (1-\epsilon)}} {5 - 4 \epsilon} \right)^{k - 1}.
\label{EM4part1}
\eeqn
Now we analyze, in similar fashion, $\EE[\tilde{I}_k^4|C_k]$ in \eqref{EM4}.
Using the assumption $0 < \epsilon < 1 / 4$ and \refL{L:intbound} we find
$$
\EE \left[ \left. \tilde{I}_k^4\,\right|\,C_k \right] \leq \frac{2^{4 \epsilon} c^4}{1 - 4 \epsilon} (R_{k - 1} - L_{k - 1})^{1 - 4 \epsilon}.
$$
Applying Lemma~\ref{boundForEdiffp} thus gives the geometric decay
\beqn
\EE\,\tilde{I}_k^4 \leq \frac{2^{4\epsilon} c^4}{1-4\epsilon} \left(\frac{2-2^{-(1-4\epsilon)}}{2-4\epsilon}\right)^{k-1}.
\label{EM4part2}
\eeqn
Therefore, it follows from \eqref{boundFor4thM}--\eqref{forBoundFor4thM} and \eqref{EM4part1}--\eqref{EM4part2} that~\eqref{showThisFor4thM} holds:
\beqnn
\sum_{n\geq 1}\EE \left(\frac{\SV}{n} - \frac{\St_n}{n} \right)^4 \leq 3 \left( \sum_{n \geq 1} n^{-2} \right) \left[ \sum_{k \geq 1} \left( \EE\,M_4(k) \right)^{1/4} \right]^4 < \infty.
\eeqnn
This completes the proof of Theorem~\ref{thm:ASConvergence}.
\end{proof}

\subsection{Computation of $\EE\,S$:\ an integral expression}\label{SS:QValIntegralExpression}

In this section we derive the following simple double-integral expression for $\EE\,S$ in terms of the
cost function~$\beta$.
\begin{theorem}\label{thm:integralExpressionForES}
For any symmetric cost function $\beta \geq 0$ we have
\begin{eqnarray*}
\EE\,S = 2 \int\!\int_{0 < u < t < 1}\!\beta(u, t)\,[(\alpha \vee t)-(\alpha \wedge u)]^{-1}\,du\,dt.
\end{eqnarray*}
\end{theorem}
\begin{proof}
Recall that $\EE\,S = \sum_{k \geq 1} \EE\,I_k$, where
$$
I_k = \int_{L_{k-1}}^{R_{k-1}}\!\beta(u, U_{\tau_k})\,du.
$$
Recall also that, for each $k$, the conditional distribution of $U_{\tau_k}$ given $L_{k-1}$ and $R_{k-1}$ is uniform($L_{k-1},R_{k-1}$).  Thus
\begin{eqnarray*}
\EE\,I_k &=& \EE\,\int_{L_{k-1}}^{R_{k-1}}\!(R_{k-1}-L_{k-1})^{-1} \int_{L_{k-1}}^{R_{k-1}}\!\beta(u,w)\,dw\,du\\
&=&\int_{0<w,u<1}\!\beta(w,u)\,\EE[(R_{k-1}-L_{k-1})^{-1} {\bf 1}(L_{k-1} < u, w < R_{k-1})]\,dw\,du\\
&=&2 \int_{0<w<u<1}\!\beta(w,u) \\
& &\hspace{.02cm} {} \times \int_{0\leq x < \alpha < y \leq 1} (y-x)^{-1} {\bf 1}(x < w < u < y) \PP(L_{k-1} \in dx,R_{k-1} \in dy)\,dw\,du.
\end{eqnarray*}
Hence
\begin{eqnarray}\label{intermediateIntegralES}
\EE\,S &=& 2\int_{0<w<u<1}\!\beta(w,u) \\
& &\  {} \times \int_{0\leq x < \alpha < y \leq 1}(y-x)^{-1} {\bf 1}(x < w < u < y)\,\nu(dx, dy)\,dw\,du\nonumber
\end{eqnarray}
where~$\nu$ is the measure
\begin{equation}
\label{nu}
\nu(dx, dy) := \sum_{k \geq 0}\PP(L_k \in dx,R_k \in dy).
\end{equation}

As established in the Appendix in \refP{P:nuformula}, one has the tractable expression
$$
\nu(dx, dy) = \delta_0(dx)\,\delta_1(dy)+ (1-x)^{-1}\,dx\,\delta_1(dy)
+ \delta_0(dx)\,y^{-1}\,dy + 2(y-x)^{-2}dx\,dy.
$$
Using this last expression, routine calculation shows that, for $0<w<u<1$,
\begin{equation}
\int_{0\leq x < \alpha < y \leq 1}(y-x)^{-1} {\bf 1}(x < w < u < y)\,\nu(dx,dy)=[(\alpha \vee u) - (\alpha \wedge w)]^{-1}.\label{partOfIntegralES}
\end{equation}
Substitute~\eqref{partOfIntegralES} into~\eqref{intermediateIntegralES} to complete the proof of the theorem.
\end{proof}

\begr
\label{rmk:finiteESforKeys}
We now let $\beta \equiv 1$ and use
\refT{thm:integralExpressionForES} to analyze the expectation of the number $K_n$ of key comparisons required by {\tt QuickVal}$(n, \alpha)$.  Then the expected value in \refT{thm:integralExpressionForES}
is
\begin{equation}
2 \int\!\int_{0 < u < t < 1}\![(\alpha \vee t)-(\alpha \wedge u)]^{-1}\,du\,dt =2[1-\alpha \ln \alpha - (1-\alpha) \ln (1-\alpha)] < \infty.\label{ESforKeyComparisons}
\end{equation}
It follows by \eqref{ESforKeyComparisons} that for $\alpha = 0$ we have
\beqnn
\lim_{n\rightarrow \infty} \EE\,K_n/n = 2,
\eeqnn
which is well known since $K_n$ in this case represents the number of key comparisons requred by {\tt QuickMin} applied to a file of $n$ keys (e.g., Mahmoud \textit{et al.}~\cite{mms1995}).  Thus we are now able to conclude that for any $\alpha$ ($0 \leq \alpha \leq 1$),  $\EE\, K_n / n$ converges to
the simple constant in~\eqref{ESforKeyComparisons}.
Also notice that we have verified the hypothesis of \refT{T:convergenceInLp} for $p=1$ (see also \ref{rmk:L1}) by \eqref{ESforKeyComparisons}, as we promised in \refR{rmk:keys} that we would.
\enr

\subsection{Computation of $\EE\,S$:\ a series expression}\label{SS:QValSeriesExpression}

We now restrict to the cost function $\beta_{\mathrm{symb}}$ and use \refT{thm:integralExpressionForES} to derive a series expression for $\EE\,S$.  In the notation of \refS{SS:source}, we have
\beqnn
\sfrac{1}{2}\,\EE\,S
 &=& \sum_{w \in \Sigma^*} \int_{\mathcal{T}_w} [(\alpha \vee t) - (\alpha \wedge u)]^{-1}\,du\,dt,
\eeqnn
which is easily obtained by noting that for $u < t$ we have
\begin{equation}
\label{betaexpression}
\beta(u, t) = \sum_{w\in \Sigma^*} {\bf 1}(a_w < u < t < b_w).
\end{equation}

Define
\beqnn
\mathcal{J}(w):=\int_{\mathcal{T}_w}\![(\alpha \vee t) - (\alpha \wedge u)]^{-1}\,du\,dt.
\eeqnn
Then routine calculation shows that
$$
\mathcal{J}(w) = p_w L\left( \left| \frac{\alpha - \mu_w}{p_w} \right| \right),
$$
where the reader should recall the definition of~$L$ near the end of \refS{SS:known}.
Thus
\begin{equation}
\label{ESseriesExpression}
\EE\,S = \sum_{w \in \Sigma^*} p_w L\left( \left| \frac{\alpha - \mu_w}{p_w} \right| \right).
\end{equation}
This last equation is in agreement with Theorem 2(i) of Vall\'{e}e \textit{et al.}~\cite{vcff2009} (see also their Figure~1).  But, unlike in~\cite{vcff2009}, our calculation requires \emph{no assumptions} of tameness, nor even that $\EE\,S$ is finite.

\section{Analysis of QuickQuant}\label{S:QQ}
Following some preliminaries in \refS{SS:QQpreliminaries}, in
Section~\ref{SS:QQp} we show that a suitably defined $\SQ / n$ converges in $L^p$ to $S$ for $1 \leq p < \infty$ provided that the cost function $\beta$ is $\epsilon$-tame with $\epsilon < 1 / p$; hence $\SQ / n$ and $\SV / n$ have the same limiting distribution provided only that the cost function~$\beta$ is 
$\epsilon$-tame for suitably small~$\epsilon$.  Granting that result for a moment, we can now relate three of the results obtained in \refS{S:QuickVal} to previously known results reviewed in \refS{SS:known}. From Remark~\ref{rmk:keys} we recover the result of \cite[Theorem~8]{gr1996} (in a cosmetically different, but equivalent, form; compare \cite[Theorem~3]{g1998}) for the limiting distribution of the number of key comparisons, and from Remark~\ref{rmk:finiteESforKeys} we recover first-moment information for the same.  Finally,
recalling that $L^1$-convergence implies convergence of means,
from~\eqref{ESseriesExpression} we recover at least the lead-order term in the asymptotics of~\cite{vcff2009} discussed at~\eqref{rhodef}.
\subsection{Preliminaries}\label{SS:QQpreliminaries}

We will closely follow the framework described in Section~\ref{S:QuickVal} for the analysis of {\tt QuickVal} and construct a random variable, call it $\SQ$, that has the distribution of the total cost required by {\tt QuickQuant} when applied to a file of $n$ keys.  Our goal is to show that, under suitable technical conditions, $\SQ / n$ converges in $L^p$ to $S$ defined at \eqref{S}.

Again, we define $\SQ$ in terms of an infinite sequence $(U_i)_{i\geq 1}$ of seeds that are
i.i.d.\ uniform$(0, 1)$.
Let $\rank$ (with $m_n / n \to \alpha$) denote our target rank for {\tt QuickQuant}.
Let $\tau_k(n)$ denote the index of the seed that corresponds to the $k$th pivot.   As in Section~\ref{SS:QValpreliminaries} we will set the first pivot index $\tau_1(n)$ to~$1$ rather than to a randomly chosen integer from $\{1, \dots, n\}$.  For $k \geq 1$, we will use $\LQkm$ and
$\RQkm$, as defined below, to denote the lower and upper bounds, respectively, of seeds of words that are eligible to be compared with the $k$th pivot.  [Notice that $\tau_k(n)$, $\LQk$, and $\RQk$ are analogous to $\tau_k$, $L_k$, and $R_k$ defined in Section~\ref{SS:QValpreliminaries}; see \eqref{pivotIndex}--\eqref{upperBound}.]
Hence we let $L_0(n):=0$ and $R_0(n):=1$, and for $k \geq 1$ we inductively define
$$
\tau_k(n) := \inf\{i \leq n: \LQkm < U_i < \RQkm\},
$$
and
\begin{eqnarray*}
\LQk &:=& {\bf 1}(\pivrank \leq \rank)\,\pSeedk + {\bf 1}(\pivrank > \rank)\,\LQkm, \\
\RQk &:=& {\bf 1}(\pivrank \geq \rank)\,\pSeedk + {\bf 1}(\pivrank < \rank)\,\RQkm
\end{eqnarray*}
if $\tau_k(n) < \infty$ but
$$
(\LQk, \RQk) := (\LQkm, \RQkm)
$$
if $\tau_k(n) = \infty$.
Here $\pivrank$ denotes the rank of the $k$th pivot seed $U_{\tau_k(n)}$ if $\tau_k(n) < \infty$ and $\rank$ otherwise.  Recall that the infimum of the empty set is~$\infty$; hence
$\tau_k(n)=\infty$ if and only if $\LQkm = \RQkm$.

Using this notation, let
\begin{eqnarray*}
\SQnk := \sum_{i:\,\tau_k(n)<i\leq n} {\bf 1}(\LQkm < U_i < \RQkm) \beta(U_i, \pSeedk)
\end{eqnarray*}
be the total cost of all comparisons (for the first~$n$ keys) with the $k$th pivot key.
Then
\begin{eqnarray}\label{SnQ}
\SQ := \sum_{k \geq 1} \SQnk
\end{eqnarray}
has the distribution of the total cost required by {\tt QuickQuant}.

Notice that the expression \eqref{SnQ} is analogous to \eqref{SV}.  In fact, we will prove the $L^p$-convergence of $\SQ / n$ to $S$ by comparing the corresponding expressions for {\tt QuickVal} and {\tt QuickQuant}.

\subsection{Convergence of $\SQ / n$ in $L^p$ for $1 \leq p < \infty$}\label{SS:QQp}

The following is our main theorem regarding {\tt QuickQuant}.

\begin{theorem}\label{thm:convergenceOfSQp}
Let $1 \leq p < \infty$.  Suppose that the cost function~$\beta$ is $\epsilon$-tame with $\epsilon < 1 / p$.
Then $\SQ / n$ converges in $L^p$ to $S$.
\end{theorem}

\begr
Note that as~$p$ increases, getting $L^p$-convergence requires the increasingly stronger condition $\epsilon < 1 / p$.  Thus we have convergence of moments of \emph{all} orders provided the source is $\gamma$-tame for \emph{every} 
$\gamma > 0$---for example, if it is g-tame as in \refR{R:g-tame}, as is true for memoryless and most Markov sources.
\enr

The proof of Theorem~\ref{thm:convergenceOfSQp} will make use of the following analogue of \refL{boundForEdiffp}, whose proof is essentially the same and therefore omitted.

\begin{lemma}\label{boundForEdiffpn}
For any $p > 0$ and $k \geq 1$ and $n \geq 1$, we have
\beqnn
\EE (R_k(n) -L_k(n))^p \leq \left(\frac{2 - 2^{-p}}{p + 1}\right)^k.
\eeqnn
\end{lemma}

\begin{proof}[Proof of Theorem~\ref{thm:convergenceOfSQp}]
Part of our strategy in proving this theorem is to compare {\tt QuickQuant} with {\tt QuickVal}.  Hence we will frequently refer to the notation established in Section~\ref{SS:QValpreliminaries} for the analysis of {\tt QuickVal}.
For each~$k$, observe that as $n \rightarrow \infty$ we have
\begin{eqnarray*}
\tau_k(n)\ \asto\ \tau_k,\ \ \pSeedk\ \asto\ U_{\tau_k},\ \ \LQk\ \asto\ L_k,\ \ \RQk\ \asto\ R_k,
\end{eqnarray*}
where $\tau_k$, $L_k$, and $R_k$, are defined in Section~\ref{SS:QValpreliminaries} [see \eqref{pivotIndex}--\eqref{upperBound}].  (In fact, in each of these four cases of convergence, the left-hand side almost surely becomes equal to its limit for all sufficiently large~$n$.)  Thus for each $k \geq 1$ we have
\begin{equation}
\label{diffvanp}
\SQnk- S_{n,k}\ \asto\ 0,
\end{equation}
where $S_{n,k}$ is defined at \eqref{Snk}; indeed, again the difference almost surely vanishes for all sufficiently large~$n$.
In proving
\refT{T:convergenceInLp}, we showed [at~\eqref{Lp}] that
\begin{eqnarray*}
\frac{S_{n,k}}{n}\ \Lpto\ I_k,
\end{eqnarray*}
where $I_k$ is defined at \eqref{I_k}, and it is somewhat easier (by means of conditional application of the strong law of large numbers, rather than the $L^p$ law of large numbers, together with Fubini's theorem) to show that
\begin{equation}
\label{strongp}
\frac{S_{n,k}}{n}\ \asto\ I_k.
\end{equation}
Combining~\eqref{diffvanp} and~\eqref{strongp}, for each $k \geq 1$ we have
\begin{eqnarray}
\frac{\SQnk}{n}\ \asto\ I_k.\label{ASconvergenceOfSQnkp}
\end{eqnarray}

What we want to show is that
\begin{eqnarray}
\frac{\SQ}{n} = \sum_{k \geq 1} \frac{\SQnk}{n}\ \Lpto\ \sum_{k \geq 1} I_k = S.\label{SQlimitp}
\end{eqnarray}
Choose any sequence $(a_k)_{k \geq 1}$  of positive numbers summing to~$1$, and let~$A$ be the probability measure on the positive integers with this probability mass function.  Then, once again using the fact that the $p$th power of the absolute value of an average is bounded by the average of $p$th powers of absolute values,
\begin{eqnarray*}
\left| \frac{\SQ}{n} - S \right|^p
&\leq& \left[ \sum_{k \geq 1} \left| \frac{\SQnk}{n} - I_k \right| \right]^p
= \left[ \sum_{k \geq 1} a_k a^{-1}_k \left| \frac{\SQnk}{n} - I_k \right| \right]^p \\
&\leq& \sum_{k \geq 1} a_k a^{-p}_k \left| \frac{\SQnk}{n} - I_k \right|^p.
\end{eqnarray*}
So for~\eqref{SQlimitp} it suffices to prove that, \emph{with respect to the product probability $\PP \times A$}, as $n \to \infty$ the sequence
$$
a^{-p}_k \left|\frac{\SQnk}{n}-I_k\right|^p
$$
converges in $L^1$ to~$0$.  What we know from~\eqref{ASconvergenceOfSQnkp} is that the sequence converges almost surely with respect to $\PP \times A$.

Now almost sure convergence together with boundedness in $L^{1 + \delta}$ are, for any $\delta > 0$, sufficient for convergence in $L^1$ because the boundedness condition implies uniform integrability (e.g.,\ Chung~\cite[Exercise 4.5.8]{c2001}).
Thus our proof is reduced to showing that, for some $q > p$, the sequence
\begin{eqnarray*}
\sum_{k\geq 1} a^{1 - q}_k\,\EE\, \left| \frac{\SQnk}{n} - I_k \right|^q
\end{eqnarray*}
is bounded in~$n$, for a suitably chosen probability mass function $(a_k)$.  Indeed, by convexity of $q$th power,
\begin{equation}
2^{1 - q} \sum_{k \geq 1} a^{1 - q}_k\,\EE\, \left|\frac{\SQnk}{n} - I_k\right|^q \leq \sum_{k \geq 1} a^{1 - q}_k\,\EE \left|\frac{\SQnk}{n}\right|^q
+ \sum_{k\geq 1} a^{1 - q}_k\,\EE\, I_k^q,\label{suffForSQlimitq}
\end{equation}
and we will show that each sum on the right-hand side of \eqref{suffForSQlimitq} is bounded in order to prove the theorem.  The value of~$q$ that we use can be any satisfying $\epsilon < 1 / q < 1 / p$.

First we recall from \refL{L:Imomentbound} that
\begin{equation}
\EE\, I_k ^q \leq \left( \frac {2^{ \epsilon} c} {1-  \epsilon} \right)^q \left( \frac{2 - 2^{ - q (1-\epsilon)}} {q (1 - \epsilon) + 1} \right)^{k - 1}, \qquad k \geq 1, \label{boundForEIkq}
\end{equation}
with geometric decay.  Thus the second sum on the right in~\eqref{suffForSQlimitq} is finite if the cost is $\epsilon$-tame with $\epsilon < 1$ and the sequence $(a_k)$ is suitably chosen not to decay too quickly.

Next we analyze $\EE|\SQnk / n|^q$ for the first sum on the right in~\eqref{suffForSQlimitq}.  Let
\begin{eqnarray*}
\nukm := |\{i: \LQkm < U_i < \RQkm,\ \tau_k(n) < i \leq n\}|.
\end{eqnarray*}
Until further notice our calculations are done only over the event $\{\nu_{k - 1}(n) > 0\}$.
Then, bounding the $q$th power of the absolute value of an average by the average of $q$th powers of absolute values,
\begin{eqnarray}
\left|\frac{\SQnk}{n}\right|^q &=& \left|\frac{1}{\nu_{k-1}(n)} \sum_{i:\,\LQkm < U_i < \RQkm} {\bf 1}(\tau_{k}(n) < i \leq n)\,\beta(U_i, \pSeedk)\right|^q\nonumber\\
& &\hspace{3cm}\times \left(\frac{\nukm}{n}\right)^q\nonumber\\
&\leq&\frac{1}{\nukm}\sum_{i:\,\LQkm < U_i < \RQkm} {\bf 1}(\tau_{k}(n) < i \leq n)\,\beta^q(U_i, \pSeedk)\label{boundForSQnkq}\\
& &\hspace{3cm}\times \left(\frac{\nukm}{n}\right)^q.\nonumber
\end{eqnarray}
Let $D_k(n)$ denote the quintuple $(\LQkm, \RQkm, \tau_{k}(n), \pSeedk, \nukm)$, and notice that, conditionally given $D_k(n)$, the $\nu_{k - 1}(n)$ values $U_i$ appearing
in~\eqref{boundForSQnkq} are i.i.d.\ unif$(L_{k - 1}(n), R_{k - 1}(n))$.
Using \eqref{boundForSQnkq}, we bound the conditional expectation of $|\SQnk / n|^q$  given $D_k(n)$.  We have
\begin{eqnarray}
\EE\left[\left.\left|\frac{\SQnk}{n}\right|^q \right| D_k(n)\right]
&\leq&[\RQkm - \LQkm]^{-1}\int_{\LQkm}^{\RQkm}\!\beta^q(u, \pSeedk)\,du\nonumber\\
& &\hspace{3cm}\times \left(\frac{\nukm}{n}\right)^q.\label{boundForCESQnkq}
\end{eqnarray}
Under $\epsilon$-tameness of~$\beta$ with $\epsilon < 1 / q$,
we find from \refL{L:intbound} that
\begin{eqnarray}
\int_{\LQkm}^{\RQkm}\!\beta^q(u, \pSeedk)\,du
\leq \frac{2^{q \epsilon} c^q}{1 - q \epsilon} [\RQkm - \LQkm]^{1-q \epsilon}.\label{boundForIntBetaq}
\end{eqnarray}
From \eqref{boundForCESQnkq}--\eqref{boundForIntBetaq}, it follows that if $\epsilon < 1 / q$, then
$$
\EE\left[\left.\left|\frac{\SQnk}{n}\right|^q \right| D_k(n)\right]
\leq \frac{2^{q \epsilon} c^q}{1 - q \epsilon}[\RQkm - \LQkm]^{q - q \epsilon}\left(
\frac{\nukm}{n (R_{k - 1}(n) - L_{k - 1}(n)) }\right)^q.
$$
Until this point we have worked only over the event $\{\nu_{k - 1}(n) > 0\}$, but now we enlarge our scope to the event $\{L_{k - 1}(n) < R_{k - 1}(n)\}$ and note that the preceding inequality holds there, as well.

Next notice that, conditionally given the triple
\beqnn
\Dt_k(n) := (L_{k - 1}(n), R_{k - 1}(n), \tau_k(n)),
\eeqnn
the  values $U_i$ with $\tau_k(n) < i \leq n$ are i.i.d.\ unif$(0, 1)$,  and so the number of them falling in the interval $(L_{k - 1}(n), R_{k - 1}(n))$ is distributed binomial$(m, t)$ with $m = n -\tau_k(n)$ and $t = R_{k - 1}(n) - L_{k - 1}(n)$, and hence (representing a binomial as a sum of independent Bernoulli random variables and applying the triangle inequality for $L^q$) has moment of order~$q$ bounded by $m^q t$. Thus
$$
\EE \left[\left.  \left(\frac{\nukm}{n(\RQkm - \LQkm)}\right)^q \right| \Dt_k(n) \right]  \leq \left[ R_{k-1}(n)-L_{k-1}(n) \right]^{1 - q},
$$
so that
$$
\EE\left[\left.\left|\frac{\SQnk}{n}\right|^q \right| \Dt_k(n)\right]
\leq \frac{2^{q \epsilon} c^q}{1 - q \epsilon}[\RQkm - \LQkm]^{1 - q \epsilon}.
$$
Since this inequality holds even when $L_{k - 1}(n) = R_{k - 1}(n)$, we can take expectations to conclude
\begin{eqnarray}
\EE \left| \frac{\SQnk}{n} \right|^q
&\leq& \frac{2^{q \epsilon} c^q}{1 - q \epsilon} \EE [\RQkm - \LQkm]^{1 - q \epsilon} \nonumber \\
&\leq& \frac{2^{q \epsilon} c^q}{1 - q \epsilon} \left( \frac{2 - 2^{ - (1 - q \epsilon)}}{2 - q \epsilon} \right)^{k - 1}, \label{ESQboundq}
\end{eqnarray}
where at the second inequality we have employed \refL{boundForEdiffpn}.

From~\eqref{boundForEIkq} and~\eqref{ESQboundq} we see that we can choose
$(a_k)$ to be the geometric distribution $a_k = (1 - \theta) \theta^{k - 1}$, $k \geq 1$, with
$$
\frac{2 - 2^{-q (1 - \epsilon)}}{q (1 - \epsilon) + 1} < \theta < 1.
$$
We then conclude
that $\sum_{k\geq 1} a_k^{1 - q}\,\EE \left| \left(\SQnk / n\right) - I_k\right|^q$ is bounded in $n$, and therefore that
$\SQ / n$ converges to $S$ in $L^p$, if the cost function is $\epsilon$-tame with
$\epsilon < 1 / p$.
\end{proof}
\smallskip

{\bf Acknowledgments.\ }The second author carried out his research as a Ph.D.\ student in the Department of Applied Mathematics and Statistics at The Johns Hopkins University.  We are grateful to Brigitte Vall\'{e}e for suggesting 
Lemmas~\ref{L:intbound}--\ref{L:Imomentbound} and for a multitude of other helpful comments, and to an anonymous referee for several helpful suggestions.

\newpage

\def\thesection{A}
\section{Appendix:\ A tractable expression for the measure~$\nu$}
\label{appendix}

The purpose of this appendix is to prove the following proposition used in the computation of $\EE\,S$ in \refS{SS:QValIntegralExpression}.

\begp
\label{P:nuformula}
With $(L_k, R_k)$ defined at \eqref{lowerBound}--\eqref{upperBound} as the interval of values eligible to be compared with the $k$th pivot chosen by {\tt QuickVal}, and with
$$
\nu(dx, dy) := \sum_{k \geq 0}\PP(L_k \in dx,R_k \in dy)
$$
as defined at~\eqref{nu}, we have
$$
\nu(dx, dy) = \delta_0(dx)\,\delta_1(dy)+ (1-x)^{-1}\,dx\,\delta_1(dy)
+ \delta_0(dx)\,y^{-1}\,dy + 2(y-x)^{-2}dx\,dy.
$$
\enp

\begin{proof}
To begin, since $L_0 := 0$ and $R_0 := 1$ we have
\begin{equation}
\PP(L_0 \in dx,R_0 \in dy)=\delta_0(dx)\,\delta_1(dy),\label{initialP}
\end{equation}
where $\delta_z$ denotes the probability measure concentrated at~$z$.
Now assume $k \geq 1$.  If $0 \leq \lambda < \alpha < \rho \leq 1$,
then
\begin{eqnarray*}
& &\PP(L_k \in dx,R_k \in dy\,|\,L_{k-1}=\lambda,R_{k-1}=\rho)\\
& &\hspace{1cm}=\delta_{\rho}(dy){\bf 1}(\lambda < x < \alpha)(\rho - \lambda)^{-1}\,dx + \delta_{\lambda}(dx){\bf 1}(\alpha < y < \rho)(\rho - \lambda)^{-1}\,dy.
\end{eqnarray*}
Hence
\begin{eqnarray}\label{recursiveP}
& &\PP(L_k \in dx,R_k \in dy)=\int\![\delta_{\rho}(dy){\bf 1}(\lambda < x < \alpha)(\rho - \lambda)^{-1} dx\\
& &\hspace{2cm}+ \delta_{\lambda}(dx){\bf 1}(\alpha < y < \rho)(\rho - \lambda)^{-1} dy]\,\PP(L_{k-1} \in d\lambda,R_{k-1} \in d\rho).\nonumber
\end{eqnarray}
We can infer [and inductively prove using~\eqref{recursiveP}] that, for $k \geq 1$,
\begin{eqnarray}
\hspace{1cm}\PP(L_k \in dx,R_k \in dy) = \delta_1(dy)f_k(x)dx + \delta_0(dx)g_k(y)dy + h_k(x,y)dx\,dy,\label{expressionForPwithFGH}
\end{eqnarray}
where
$$
f_1(x) = {\bf 1}(0 \leq x < \alpha), \quad g_1(y) = {\bf 1}(\alpha < y \leq 1), \quad h_1(x, y) = 0,
$$
and, for $k \geq 2$,
\begin{eqnarray}
\hspace{1cm} f_k(x)&=&{\bf 1}(0 \leq x < \alpha)\int\!{\bf 1}(0 \leq \lambda < x) (1-\lambda)^{-1}f_{k-1}(\lambda)\,d\lambda,\label{recursiveF}\\
g_k(y)&=&{\bf 1}(\alpha < y \leq 1)\int\!{\bf 1}(y < \rho \leq 1) \rho^{-1}g_{k-1}(\rho)\,d\rho,\label{recursiveG}\\
h_k(x,y)&=&{\bf 1}(0 \leq x < \alpha < y \leq 1) \Big[ (1-x)^{-1}f_{k-1}(x) + y^{-1}g_{k-1}(y)\label{recursiveH}\\
& &\hspace{.5cm}{} + \int\!{\bf 1}(0 \leq \lambda < x)(y-\lambda)^{-1}h_{k-1}(\lambda,y)\,d\lambda\nonumber\\
& &\hspace{.5cm}{} + \int\!{\bf 1}(y < \rho \leq 1)(\rho-x)^{-1}h_{k-1}(x,\rho)\,d\rho \Big].\nonumber
\end{eqnarray}
Henceforth suppose $0 \leq x < \alpha < y \leq 1$.  From~\eqref{recursiveG} we obtain
\begin{equation}
\label{gexplicit}
g_k(y) = \frac{(-\ln y)^{k-1}}{(k-1)!}, \quad k \geq 1,
\end{equation}
whence
\begin{equation}
\sum_{k \geq 1} g_k(y) =y^{-1}.\label{sumOfG}
\end{equation}
By recognizing symmetry between \eqref{recursiveF} and \eqref{recursiveG}, we also find
\begin{equation}
\label{fexplicit}
f_k(x) = \frac{[-\ln(1-x)]^{k-1}}{(k-1)!}, \quad k \geq 1,
\end{equation}
and so
\begin{equation}
\sum_{k \geq 1} f_k(x) = (1-x)^{-1}.\label{sumOfF}
\end{equation}
In order to compute $\sum_{k \geq 1}h_k(x,y)$, we consider the generating function
\begin{eqnarray}\label{Hxyz}
H(x,y,z):=\sum_{k \geq 1}h_k(x,y)\,z^k.
\end{eqnarray}
From~\eqref{recursiveH},
\begin{eqnarray}
H(x,y,z) &=& z\left[(1-x)^{-1}\sum_{k \geq 1}f_k(x)\,z^{k} + y^{-1}\sum_{k \geq 1}g_k(y)\,z^k\right.\nonumber\\
& &\left. \qquad {} + \int_0^x\!(y-\lambda)^{-1}H(\lambda,y,z)\,d\lambda + \int_y^1\!(\rho - x)^{-1}H(x,\rho,z)\,d\rho\right].\label{integralEqn}
\end{eqnarray}
Using this integral equation, we will show via a series of lemmas culminating in \refL{L:last} that
\begin{eqnarray}
H(x, y) := H(x, y, 1) =\sum_{k \geq 1}h_k(x,y)\mathrm{\ equals\ }2(y-x)^{-2}.\label{Hxy}
\end{eqnarray}
Combining
equations
\eqref{expressionForPwithFGH}, \eqref{sumOfG}, \eqref{sumOfF}, and~\eqref{Hxy}, we obtain the desired expression for~$\nu$.
\end{proof}

Throughout the remainder of this appendix, whenever we refer to $H(x, y)$ we tacitly suppose that $0 \leq x < \alpha < y \leq 1$.

\begin{lemma}\label{lemma1}
$H(x,y) < \infty$ almost everywhere.
\end{lemma}

\begin{proof}
We revisit Remarks~\ref{rmk:keys} and~\ref{rmk:finiteESforKeys}
and consider the number of key comparisons required by {\tt QuickVal}$(n, \alpha)$.  As shown at \eqref{ESforKeyComparisons}, we have $\EE\, S < \infty$ in this case. On the other hand, with $\beta \equiv 1$,
from
\eqref{intermediateIntegralES}--\eqref{nu}, \eqref{initialP}, \eqref{expressionForPwithFGH},
and \eqref{sumOfG}--\eqref{sumOfF}, we have
\begin{eqnarray*}
\EE\,S &=& 2 \int_{0 < w < u < 1} \left[1 + \int_{0 \leq x < \alpha} (1 - x)^{-1}\,{\bf 1}(x < w)\,dx + \int_{\alpha < y \leq 1} y^{-1}\,{\bf 1}(y > u)\,dy\right. \\
& & \quad \left.{} + \int_{0 \leq x < \alpha < y \leq 1} (y - x)^{-1}\,{\bf 1}(x < w < u < y)\,H(x, y)\,dx\,dy \right] dw\,du.
\end{eqnarray*}
Thus $H(x,y) < \infty$ almost everywhere.
\end{proof}

The next lemma establishes monotonicity properties of $H(x, y)$.
\begin{lemma}\label{lemma2}
$H(x,y)$ is increasing in~$x$ and decreasing in~$y$.
\end{lemma}
\begin{proof}
\noindent For each $k \geq 1$, we see from~\eqref{fexplicit} that $f_k(x)$ is increasing in~$x$ and
from~\eqref{gexplicit} that $g_k(y)$ is decreasing in~$y$.  Since $h_1 \equiv 0$, it follows by induction on~$k$ from~\eqref{recursiveH} that $h_k(x,y)$ is increasing in~$x$ and decreasing in~$y$ for each~$k$.  Thus $H(x,y)=\sum_{k \geq 1} h_k(x, y)$ enjoys the same monotonicity properties.
\end{proof}

\begin{lemma}\label{lemma2.5}
$H(x, y) < \infty$ for all~$x$ and~$y$.
\end{lemma}

\begin{proof}
This is immediate from Lemmas~\ref{lemma1}--\ref{lemma2}.
\end{proof}

\begin{lemma}\label{lemma3}
The generating function $H(x, y, z)$ at~\eqref{Hxyz}, is {\rm (}with $h_0 :\equiv 0${\rm )}\ the unique power-series solution $\Ht(x, y, z) = \sum_{k \geq 0} \hti_k(x, y) z^k$ {\rm (}in $0 \leq z \leq 1${\rm )}\  to the integral
equation~\eqref{integralEqn} such that $0 \leq \hti_k(x, y) \leq h_k(x, y)$ for all $k, x, y$.
\end{lemma}

\begin{proof}
We have already seen that~$H$ is such a solution.  Conversely, if~$\Ht$ is such a solution, then equating coefficients of $z^k$ in the integral equation [which is valid because we know by Lemma~\ref{lemma2.5}
that $H(x, y, z)$, and hence also
$\Ht(x, y, z)$, is finite for $0 \leq z \leq 1$] we find that the functions $\hti_k(x, y)$ satisfy $\hti_k \equiv 0$ for $k = 0, 1$ and the recurrence relation~\eqref{recursiveH} for $k \geq 2$.  It then follows by induction that $\hti_k(x, y) = h_k(x, y)$ for all $k, x, y$.
\end{proof}

Next we let $H_0(x,y,z) :\equiv 0$ and, for $0 \leq z \leq 1$,  inductively define $H_n(x,y,z)$ by applying successive substitutions to the integral equation~\eqref{integralEqn}; that is, for each $n \geq 1$ we define
\beqn
H_n(x,y,z)&:=& z\left[(1-x)^{-1}\sum_{k \geq 1}f_k(x)\,z^{k} + y^{-1}\sum_{k \geq 1}g_k(y)\,z^k\right.\nonumber\\
& &\left. {} + \int_0^x\!(y-\lambda)^{-1}H_{n-1}(\lambda,y,z)\,d\lambda + \int_y^1\!(\rho - x)^{-1}H_{n-1}(x,\rho,z)\,d\rho\right].\label{Hnxy}
\eeqn
Let
$[z^k]\, H_n(x,y,z)$ denote the coefficient of $z^k$ in $H_n(x, y, z)$.

\begin{lemma}\label{lemma4}
For each $k\geq 1$, $[z^k]\,H_n(x, y, z)$ is nondecreasing in $n \geq 0$.
\end{lemma}
\begin{proof}
The inequality $[z^k]\,H_n(x, y, z) \geq [z^k]\,H_{n - 1}(x, y, z)$
is proved easily by induction on $n \geq 1$.
\end{proof}

According to the next lemma, $H$ dominates each $H_n$.
\begin{lemma}\label{lemma5}
For all $n \geq 0$ and $k \geq 1$ we have
\beqn
0 \leq [z^k]\,H_n(x,y,z) \leq h_k(x, y, z).\label{inequalitiesInLemma5}
\eeqn
\end{lemma}
\begin{proof}
Lemma~\ref{lemma4}
establishes the first inequality, and the second is proved easily by induction on $n$.
\end{proof}

Lemmas~\ref{lemma3}--\ref{lemma5} lead to the following lemma:
\begin{lemma}\label{lemma6}
For $0 \leq x < \alpha < y \leq 1$ and $0 \leq z \leq 1$ we have
\beqnn
H_n(x, y, z) \uparrow H(x, y, z)\ \mbox{as }n \uparrow \infty.
\eeqnn
\end{lemma}

\begin{proof}
Recalling Lemmas~\ref{lemma4}--\ref{lemma5}, define $\Ht(x, y, z)$ to be the power series in~$z$ with coefficient of $z^k$ equal to $\hti_k(x, y) := \lim_{n \uparrow \infty} [z^k]\,H_n(x, y, z)$, which satisfies $0 \leq \hti_k(x, y) \leq h_k(x, y)$.
On the other hand, $\Ht$ satisfies the integral equation~\eqref{integralEqn} by applying the monotone convergence theorem to~\eqref{Hnxy}.  Thus it follows from Lemma~\ref{lemma3} that $\Ht=H$.  Finally, another application of the monotone convergence theorem shows that~$\Ht(x, y, z) = \lim_{n \uparrow \infty} H_n(x, y, z)$.
\end{proof}

Our next lemma, when combined with the preceding one, immediately leads to inequality in one direction in~\eqref{Hxy}.

\begin{lemma}\label{lemma7}
For $0 \leq x < \alpha < y \leq 1$ and all $n \geq 0$,
\beqnn
H_{n}(x,y,1) \leq 2(y-x)^{-2}.
\eeqnn
\end{lemma}

\begin{proof}
We will prove this lemma by induction on $n$, starting with
\beqnn
H_0(x,y) = 0 \leq 2(y-x)^{-2}.
\eeqnn
Suppose that the claim holds for $n - 1$.  Then from \eqref{Hnxy}, \eqref{sumOfG}, and~\eqref{sumOfF} we have
\beqnn
H_{n}(x,y,1) &\leq& (1-x)^{-2}+y^{-2}+(y-x)^{-2}-y^{-2}+ (y-x)^{-2}-(1-x)^{-2}\\
&=& 2(y-x)^{-2}. \hspace{3in}\qed
\eeqnn
\noqed
\end{proof}

\bigskip
Finally we are ready to prove~\eqref{Hxy}.

\begl
\label{L:last}
For $0 \leq x < \alpha < y \leq 1$,
\beqnn
H(x,y,1) = 2(y-x)^{-2}.
\eeqnn
\enl

\begin{proof}
Define
\beqnn
\widehat{H}(x,y):=2(y-x)^{-2}-H(x,y).
\eeqnn
Then to prove the desired equality it suffices to show that for any integer $r \geq 0$ we have
\beqn
0 \leq \widehat{H}(x,y) \leq \mbox{$(\frac{2}{3})^{r}$} \times 2(y-x)^{-3}.\label{inequalitiesForHhat}
\eeqn
As remarked earlier, the nonnegativity of $\Hh$ follows from Lemmas~\ref{lemma6}--\ref{lemma7}.  We prove the upper bound on~$\Hh$ in \eqref{inequalitiesForHhat} by induction on $r$.   The bound clearly holds for $r=0$.  Notice that by substituting $z = 1$ and $H(x,y)=2(y-x)^{-2}-\widehat{H}(x,y)$ into the integral equation~\eqref{integralEqn} we find
\beqnn
\widehat{H}(x,y)&=&2(y-x)^{-2}-H(x,y)\\
&=&2(y-x)^{-2} - \left\{(1-x)^{-2} + y^{-2}+\int_0^x\!(y-\lambda)^{-1}[2(y-\lambda)^{-2} - \widehat{H}(\lambda,y)]\,d\lambda\right.\\
& &\hspace{2.5cm}\left.{} +\int_y^1\!(\rho - x)^{-1}[2(\rho-x)^{-2} - \widehat{H}(x,\rho)]\,d\rho\right\}\\
&=&\int_0^x\!(y-\lambda)^{-1}\widehat{H}(\lambda,y)\,d\lambda + \int_y^1\!(\rho - x)^{-1}\widehat{H}(x,\rho)\,d\rho.\label{Hhat}
\eeqnn
Thus if we assume that the upper bound in~\eqref{inequalitiesForHhat} holds for $r - 1$, then
\beqnn
\widehat{H}(x,y) &\leq& \left( \frac{2}{3} \right)^{r - 1} \times 2\left[\int_0^x\!(y-\lambda)^{-4}d\lambda + \int_y^1\!(\rho - x)^{-4}\,d\rho\right]\\
&\leq& (\sfrac{2}{3})^r \times 2(y-x)^{-3}.
\eeqnn
Hence \eqref{inequalitiesForHhat} holds for any nonnegative integer $r$.
\end{proof}

\bibliographystyle{plain}
\bibliography{references}

\begin{thebibliography}{10}

\bibitem{b1995}
P.~Billingsley.
\newblock {\em Probability and Measure}.
\newblock Wiley Interscience, New York, 3rd edition, 1995.

\bibitem{c2001}
K.~L. Chung.
\newblock {\em A Course in Probability Theory}.
\newblock Academic Press, London, 3rd edition, 2001.

\bibitem{cfv2001}
J.~Cl{\'e}ment, P.~Flajolet, and B.~Vall{\'e}e.
\newblock Dynamical sources in information theory: a general analysis of trie
  structures.
\newblock {\em Algorithmica}, 29(1-2):307--369, 2001.
\newblock Average-case analysis of algorithms (Princeton, NJ, 1998).

\bibitem{d1984}
L.~Devroye.
\newblock Exponential bounds for the running time of a selection algorithm.
\newblock {\em Journal of Computer and System Sciences}, 29:1--7, 1984.

\bibitem{d2001}
L.~Devroye.
\newblock On the probablistic worst-case time of ``{F}ind''.
\newblock {\em Algorithmica}, 31:291--303, 2001.

\bibitem{fjfeb2010}
J.~A. Fill.
\newblock Distributional convergence for the number of symbol comparisons used
  by {Q}uicksort.
\newblock {\em Annals of Applied Probability}.
\newblock To appear, 2012.

\bibitem{fj2004}
J.~A. Fill and S.~Janson.
\newblock The number of bit comparisons used by {Q}uicksort: An average-case
  analysis.
\newblock {\em Electronic Journal of Probability}, 17, article 43
  (electronic):1--22, 2012.

\bibitem{fn2009}
J.~A. Fill and T.~Nakama.
\newblock Analysis of the expected number of bit comparisons required by
  {Q}uickselect.
\newblock {\em Algorithmica}, 58:730--769, 2010.

\bibitem{g1998}
R.~Gr{\"{u}}bel.
\newblock Hoare's selection algorithm:\ a {M}arkov chain approach.
\newblock {\em Journal of Applied Probability}, 35:36--45, 1998.

\bibitem{gr1996}
R.~Gr{\"{u}}bel and U.~R{\"{o}}sler.
\newblock Asymptotic distribution theory for {H}oare's selection algorithm.
\newblock {\em Advances in Applied Probability}, 28:252--269, 1996.

\bibitem{h1961}
C.~A.~R. Hoare.
\newblock Find (algorithm 65).
\newblock {\em Communications of the ACM}, 4:321--322, 1961.

\bibitem{ht2002}
H.~Hwang and T.~Tsai.
\newblock Quickselect and the {D}ickman function.
\newblock {\em Combinatorics, Probability and Computing}, 11:353--371, 2002.

\bibitem{k1972}
D.~E. Knuth.
\newblock Mathematical analysis of algorithms.
\newblock In {\em Information Processing 71 (Proceedings of IFIP Congress,
  Ljubljana, 1971)}, pages 19--27. North-Holland, Amsterdam, 1972.

\bibitem{lm1996}
J.~Lent and H.~M. Mahmoud.
\newblock Average-case analysis of multiple {Q}uickselect: An algorithm for
  finding order statistics.
\newblock {\em Statistics and Probability Letters}, 28:299--310, 1996.

\bibitem{mms1995}
H.~M. Mahmoud, R.~Modarres, and R.~T. Smythe.
\newblock Analysis of {Q}uickselect: An algorithm for order statistics.
\newblock {\em RAIRO Informatique Th\'{e}orique et Applications}, 29:255--276,
  1995.

\bibitem{ms1998}
H.~M. Mahmoud and R.~T. Smythe.
\newblock Probabilistic analysis of multiple {Q}uickselect.
\newblock {\em Algorithmica}, 22:569--584, 1998.

\bibitem{pa1995}
Volkert Paulsen.
\newblock The moments of {FIND}.
\newblock {\em J. Appl. Probab.}, 34(4):1079--1082, 1997.

\bibitem{p1995}
H.~Prodinger.
\newblock Multiple {Q}uickselect---{H}oare's {F}ind algorithm for several
  elements.
\newblock {\em Information Processing Letters}, 56:123--129, 1995.

\bibitem{r1989}
M.~R{\'{e}}gnier.
\newblock A limiting distribution of {Q}uicksort.
\newblock {\em RAIRO Informatique Th\'{e}orique et Applications}, 23:335--343,
  1989.

\bibitem{r1991}
U.~R{\"{o}}sler.
\newblock A limit theorem for {Q}uicksort.
\newblock {\em RAIRO Informatique Th\'{e}orique et Applications}, 25:85--100,
  1991.

\bibitem{rr2001}
U.~R{\"{o}}sler and L.~R{\"{u}}schendorf.
\newblock The contraction method for recursive algorithms.
\newblock {\em Algorithmica}, 29(1):3--33, 2001.

\bibitem{r2002}
S.~Ross.
\newblock {\em A First Course in Probability}.
\newblock Prentice Hall, Upper Saddle River, NJ, 6th edition, 2002.

\bibitem{vcff2009}
B.~Vall\'{e}e, J.~Cl\'{e}ment, J.~A. Fill, and P.~Flajolet.
\newblock The number of symbol comparisons in {Q}uicksort and {Q}uickselect.
\newblock {\em {\rm To appear in} 36th International Colloquium on Automata,
  Languages and Programming (ICALP 2009)}, 2009.

\end{thebibliography}
\end{document}